\documentclass[a4paper]{article}
\usepackage{amsmath}
\usepackage{amssymb}
\usepackage{amsthm}
\usepackage{mathrsfs}
\usepackage{wasysym}
\usepackage{bbm}		
\usepackage[inline]{enumitem}
\usepackage{nccmath}	
\usepackage{mathtools}
\usepackage{braket}		
\usepackage[normalem]{ulem}		
\usepackage{nicefrac}
\usepackage{cancel}
\usepackage{xcolor}
\usepackage[a4paper, left=2.7cm, right=2.7cm, top=3.5cm, bottom=5cm]{geometry}
\usepackage[pdftex=true,hyperindex=true,pdfborder={0 0 0}]{hyperref}
\usepackage[nottoc]{tocbibind}	
\usepackage{doi}
\usepackage[numbers]{natbib}
\usepackage{cleveref}
\usepackage{tikz}
\usetikzlibrary{shapes,arrows,automata,matrix,fit,patterns}
\usepackage{soul} 
\usepackage{autonum} 

\theoremstyle{plain}
\newtheorem{theorem}{Theorem}[section]
\newtheorem{lemma}[theorem]{Lemma}
\newtheorem{corollary}[theorem]{Corollary}
\newtheorem{proposition}[theorem]{Proposition}
\newtheorem{question}[theorem]{Question}
\newtheorem{conjecture}[theorem]{Conjecture}
\newtheorem{problem}[theorem]{Problem}

\theoremstyle{definition}
\newtheorem{observation}[theorem]{Observation}
\newtheorem{definition}[theorem]{Definition}
\newtheorem{remark}[theorem]{Remark}

\global\long\def\ZZ {\mathbb{Z}}
\global\long\def\RR {\mathbb{R}}
\global\long\def\QQ {\mathbb{Q}}
\global\long\def\NN {\mathbb{N}}
\global\long\def\F{\mathcal{F}}%
\newcommand{\define}[1]{\textbf{#1}}

\global\long\def\M{\mathcal{\mathfrak{M}}}

\newcommand{\meff}[1]{\M_{\texttt{EFF}}(#1)}
\newcommand{\msft}[1]{\M_{\texttt{SFT}}(#1)}
\newcommand{\msof}[1]{\M_{\texttt{SOF}}(#1)}

\title{%
	Medvedev degrees of subshifts on groups
}

\author{Sebasti\'an Barbieri and Nicanor Carrasco-Vargas}

\newcommand{\Addresses}{{
		\bigskip
		
		\hskip-\parindent   S.~Barbieri, \textsc{Departamento de Matem\'{a}tica y ciencia de la computaci\'{o}n, Universidad de Santiago de Chile, Santiago, Chile.}\par\nopagebreak
		\textit{E-mail address}: \texttt{sebastian.barbieri@usach.cl}
		
		\medskip
		
		\hskip-\parindent   N.~Carrasco-Vargas, \textsc{Mathematics faculty, Jagiellonian University, Krak\'ow, Poland.}\par\nopagebreak
		\textit{E-mail address}: \texttt{nicanor.vargas@uj.edu.pl}
}}

\begin{document}
\maketitle


\begin{abstract}
    The Medvedev degree of a subshift is a dynamical invariant of computable origin that can be used to compare the complexity of subshifts that contain only uncomputable configurations. We develop theory to describe how these degrees can be transferred from one group to another through algebraic and geometric relations, such as quotients, subgroups, translation-like actions and quasi-isometries.
    
    We use the aforementioned tools to study the possible values taken by this invariant on subshifts of finite type on some finitely generated groups. We obtain a full classification for some classes, such as virtually polycyclic groups and branch groups with decidable word problem. We also show that all groups which are quasi-isometric to the hyperbolic plane admit SFTs with nonzero Medvedev degree. Furthermore, we provide a classification of the degrees of sofic subshifts for several classes of groups. 
    
    

		\medskip
		
		

	\end{abstract}
 
\section{Introduction}


The study of subshifts of finite type (SFTs) on groups other than $\mathbb{Z}$ has been revolutionized by a number of results that relate their dynamical aspects in strong ways with recursion-theoretical notions. For instance, the topological entropies of SFTs on $\mathbb{Z}^{2}$ and other amenable groups, have been characterized as the nonnegative $\Pi_1^0$ real numbers~\cite{hochman_characterization_2010,barbieri_entropies_2021,bartholdi_salo_shift_lamplighter_2024}. We also recall that the first aperiodic SFT on $\mathbb{Z}^{2}$ came along with a proof of the undecidability of the emptiness problem for $\mathbb{Z}^{2}$-SFTs~\cite{berger_undecidability_1966}. These and other related results emphasize the fact that understanding recursion-theoretical aspects of SFTs on groups is paramount in order to properly describe their dynamics.  

Hanf and Myers proved the existence of a $\mathbb{Z}^2$-SFT on some finite alphabet $A$ such that every configuration $x\colon \mathbb{Z}^2\to A$ is an uncomputable function~\cite{hanf_Nonrecursive_1974,myers_Nonrecursive_1974}. This result can be further refined using the language of Medvedev degrees. Informally, the Medvedev degree of a set measures the algorithmic complexity of computing one of its elements, and can be used to meaningfully distinguish sets without computable points. A formal definition of Medvedev degrees is given in~\Cref{section:Medvedev}.

In this work we propose a systematic study of Medvedev degrees as a topological conjugacy invariant for subshifts on groups. These degrees form a lattice and capture the algorithmic complexity of a subshift. For instance, the Medvedev degree of a subshift is zero exactly when it has computable configurations. As a measure of complexity, the Medvedev degree of a subshift shares some properties with topological entropy for amenable groups: it does not increase under factor maps, and behaves nicely with direct products and disjoint unions. These properties hold for subshifts on any finitely generated group, with no assumption on the complexity of its word problem or amenability. 

The main goal of this project is to study the following classification problem:

\begin{problem}\label{problem}
    Given a finitely generated group $G$, what is the class of Medvedev degrees of $G$-SFTs?
\end{problem}

Remarkable effort has been put into classifying the possible values of topological entropy for SFTs. Lind~\cite{Lind1984} provided a full classification of the entropies of $\ZZ$-SFTs, Hochman and Meyerovitch provided a classification for $\ZZ^d$-SFTs for $d\geq 2$, the first author extended their classification to several classes of amenable groups~\cite{barbieri_entropies_2021}, while Bartholdi and Salo recently extended it to Baumslag-Solitar and Lamplighter groups~\cite{bartholdi_salo_shift_lamplighter_2024}. This work can be thought of as an analogous effort for Medvedev degrees.

\subsubsection*{State of the art}

There is a complete answer for $\ZZ^d$, $d\geq 1$. Indeed, it is well-known that every nonempty $\mathbb{Z}$-SFT has finite orbits which implies that its Medvedev degree is zero. On the other hand, Simpson proved that the class of Medvedev degrees of nonempty SFTs on $\mathbb{Z}^{d}$, $d\geq2$ is the class of $\Pi_{1}^{0}$ degrees \cite{simpson_Medvedev_2014}. This result implies and refines greatly the result of Hanf and Myers~\cite{hanf_Nonrecursive_1974,myers_Nonrecursive_1974}. In addition, nonempty SFTs with nonzero Medvedev degree have been constructed on Baumslag Solitar groups $\operatorname{BS}(n,m)$, $n,m\geq 1$ \cite{aubrun_Tiling_2013}, and the lamplighter group~\cite{bartholdi_salo_shift_lamplighter_2024}. 

We also mention that for the more general class of effective subshifts on $\ZZ$, Medvedev degrees have been characterized as those $\Pi_1^0$ degrees by Miller~\cite{miller_Two_2012}. This result was later extended to all infinite finitely generated groups with decidable word problem by the second author~\cite{carrasco-vargas_geometric_2023_}. We also mention that Medvedev degrees of subshifts have been considered in \cite{MichaelHochman2009DiscreteandContinuousDynamicalSystems,ballier2013universality} in relation to the existence of systems that are universal for factor maps~\cite{Nicanor_rokhlik_2026}.   

\subsubsection*{Main results}

Given a finitely generated group $G$, we denote by $\msft{G}$ the class of Medvedev degrees of $G$-SFTs. First we observe that if $G$ is recursively presented, then $\msft{G}$ must be contained in the class of $\Pi_1^0$ degrees (see~\Cref{prop:msft-and-msoff-are-contained-in-Pi-1-degrees}). In order to study~\Cref{problem}, we study the behavior of $\msft{G}$ and some variants with respect to different group theoretical relations, such as subgroups (\Cref{cor:subgroups_decidable_membership}), commensurability (\Cref{prop:commensurable_same_degrees}), quotients (\Cref{prop:quotients_fg}), translation-like actions (\Cref{prop:medvedev-degrees-and-translation-like-actions,prop:translation-like-action-subshift-is-effective} and~\Cref{prop:translation-like-actions-and-medvedev-degrees}), and quasi-isometries (\Cref{lem:QI_medvedev} and~\Cref{coro_QI}).

We provide a full classification of $\msft{G}$ for virtually polycyclic groups (\Cref{thm:polycyclic}), namely, we show that for a virtually polycyclic group $G$ which is virtually cyclic then $\msft{G}$ consists of only the zero degree, and otherwise $\msft{G}$ is the class of all $\Pi_1^0$ Medvedev degrees. We also show that $\msft{G}$ is this latter class for direct products of infinite groups with decidable word problem (\Cref{thm:direct_products_clasification_medvedev_degrees}), and branch groups with decidable word problem (\Cref{cor:branch_groups}). Without hypotheses on the word problem, we are still able to show that all direct products of groups and branch groups admit SFTs with nonzero Medvedev degree (\Cref{thm:direct_products_nontrivial_medvedev_degrees}). Furthermore, we prove that all groups which are quasi-isometric to the hyperbolic plane admit SFTs with nonzero Medvedev degree (\Cref{thm:hyperbolic_qi_medvedev}). 

For some groups where we are not able to compute $\msft{G}$ we can still prove that the class $\msof{G}$ of Medvedev degrees of sofic subshifts equals the class of $\Pi_1^0$ Medvedev degrees. This is done coupling our results with existing results about Medvedev degrees of effective subshifts, and existing simulation results. Simulation results relate sofic subshifts on one group with effective subshifts on a different group (see \Cref{subsec:simulation} for a precise definition).

We prove that $\msof{G}$ equals the class of $\Pi_1^0$ Medvedev degrees for all infinite finitely generated groups with decidable word problem that simulate some other group verifying the same hypotheses (\Cref{prop:simulation}). This result covers self-simulable groups with decidable word problem, such as Thompson's $V$, $\operatorname{GL}_n(\ZZ)$, $\operatorname{SL}_n(\ZZ)$, $\operatorname{Aut}(F_n)$ and $\operatorname{Out}(F_n)$ for $n\geq 5$, where $F_n$ denotes the free groups on $n$ generators~\cite{Barbieri_Sablik_Salo_2021}. By the results in~\cite{bartholdi_salo_shift_lamplighter_2024}, this also applies to the Baumslag-Solitar groups $\operatorname{BS}(1,n)$, $n\geq 1$, and the lamplighter group.   

\subsubsection*{Open problems}

Let us note that a finitely generated group admits a nonempty SFT with nonzero Medvedev degree if and only if it admits a sofic subshift with nonzero Medvedev degree. However, this does not necessarily imply that $\msft{G}=\msof{G}$. This raises the question of whether every sofic $G$-subshift admits an SFT extension with equal Medvedev degree (\Cref{question:equal_degree_extension}). This is a Medvedev-degree version of the question of whether every sofic subshift on an amenable group admits an equal entropy SFT extension (see \cite[Problem 9.4]{hochman_characterization_2010}).

The classical observation that $\ZZ$-SFTs can only achieve the zero degree can be easily generalized to virtually free groups (\Cref{prop:virtually_free_have_0_degree}). In every other finitely generated and recursively presented group where the set $\msft{G}$ is known, it is the set of all $\Pi_1^0$ degrees. This leads us to conjecture (\Cref{conjecture_allSFTSarePi10completeclasses}) that for every infinite, finitely generated and recursively presented group which is not virtually free, then $\msft{G}$ is the set of all $\Pi_1^0$ Medvedev degrees.

For finitely generated groups which are not recursively presented, the $\Pi_1^0$ bound may not hold. In that context we pose the less ambitious conjecture that if $G$ is finitely generated and not virtually free, then there exists a nonempty $G$-SFT with nontrivial Medvedev degree (\Cref{conj:uncomputableconfig}).

Overall, we expect that positive solutions to any of these two conjectures would be rather distant with the existing techniques. \Cref{conj:uncomputableconfig} would imply a positive solution to Carroll and Penland's conjecture about groups admitting weakly aperiodic SFTs \cite{carroll_Periodic_2015}, and to Ballier and Stein's conjecture about groups with undecidable domino problem~\cite{ballier_domino_2018} (see \Cref{Medvedev-conjecture-is-pulent}). 

\subsubsection*{Acknowledgements}
The authors acknowledge that this project was partly developed while attending CIRM's thematic month ``Discrete Mathematics \& Computer Science: Groups, Dynamics, Complexity, Words''. S. Barbieri was supported by ANID 1240085 FONDECYT regular. N. Carrasco-Vargas was supported by ANID 21201185, ANID/CENIA FB210017, MSCA 731143, and a grant from the Priority Research Area SciMat under the Strategic Programme Excellence Initiative at Jagiellonian University.
\section{Preliminaries}\label{sec:preliminaries}

\subsection{Finitely generated groups}

Let $G$ be a finitely generated group and $S$ be a finite symmetric set of generators. For a word $w\in S^*$ we write $\underline{w}$ for the corresponding element of $G$. We denote $F(S)$ the free group generated by $S$.

Given a finitely generated subgroup $H \leqslant G$, the membership problem of $H$ in $G$ is the language \[ \texttt{WP}_S(G,H) = \{ w \in S^* : \underline{w} \in H\}.     \]

In the case where $H$ is the trivial group the set $\texttt{WP}_S(G) = \texttt{WP}_S(G,1)$ is called the word problem of $G$. A group is called recursively presented if $\texttt{WP}_S(G)$ is a recursively enumerable language for some (equivalently every) generating set $S$.  
%

\begin{definition}
    Let $H$ and $G$ be finitely generated by $T$ and $S$, respectively. A map $f\colon H \to G$ is \define{computable} if there exists a computable map $\widehat{f}\colon T^* \to S^*$ such that,  whenever $u$ is a word representing $h\in H$, $\widehat{f}(u)$ is a word representing $f(h)$ in $G$.
\end{definition} 

It is easily seen that the computability  of $f$ does not depend of the generators chosen. Observe that every  homomorphism between finitely generated groups is computable: we can always take $\widehat f$  as a monoid homomorphism $T^\ast\to S^\ast$, and such a map is computable. 

\subsection{Shift spaces and morphisms}\label{subsec:shift_spaces}

    Let $A$ be an alphabet set and $G$ be a group. The \define{full $G$-shift} is the set $A^{G} = \{ x\colon G \to A\}$ equipped with the prodiscrete topology and with the left \define{shift} action $G \curvearrowright A^{G}$ by left multiplication given by 
\[ (g x)(h) = x(g^{-1}h) \qquad \mbox{  for every } g,h \in G 
\mbox{ and } x \in A^G. \]

The elements $x \in A^G$ are called \define{configurations}. For a finite set $F\subset G$, a \define{pattern} with support $F$ is an element $p \in A^F$. We denote the cylinder generated by $p$ by $[p] = \{ x \in A^{G} : x|_F = p \}$ and note that the cylinders are a clopen base for the prodiscrete topology on $A^{G}$. 
 
\begin{definition}
	A $G$-\define{subshift} is a $G$-invariant and closed subset $X \subset A^G$.
\end{definition}

When the context is clear, we will drop the $G$ and plainly speak of a subshift. Equivalently, $X$ is a $G$-subshift if and only if there exists a set of forbidden patterns $\F$ such that \[X= X_{\F} = \{ x \in A^{G} : gx\notin [p] \mbox{ for every } g \in G, p \in \F  \}.\]

Given two subshifts $X\subset A^G$ and $Y \subset B^G$, a map $\phi\colon X\to Y$ is called a \define{morphism} if it is continuous and $G$-equivariant. By the Curtis-Hedlund-Lyndon theorem (see~\cite[Theorem 1.8.1]{ceccherini-silberstein_Cellular_2010}) a map $\phi\colon X \to Y$ is a morphism if and only if there is a finite $F\subset G$ and $\Phi \colon A^F\to B$ such that $\phi(x)(g) = \Phi((g^{-1} x)|_{F})$ for every $x \in X$, $g \in G$. A morphism $\phi \colon X \to Y$ is a \define{topological factor map} if it is surjective and a \define{topological conjugacy} if it is bijective.

Three countable classes of subshifts are of special interest in the literature, their common theme is that they can be defined using finite information. 

\begin{definition}
    A subshift $X$ is of \define{finite type} (SFT) if there exists a finite set of forbidden patterns $\mathcal{F}$ for which $X = X_{\mathcal{F}}$.
\end{definition}

\begin{definition}
     A subshift $Y$ is \define{sofic} if there exists an SFT $X$ and a topological factor map $\phi \colon X \to Y$.
\end{definition}

The third class of interest is that of effective subshifts. Intuitively, these are the subshifts that can be described through a Turing machine. Their definition on finitely generated groups is subtle when the underlying group is not recursively presented, thus we will delay its definition to the next section.

\subsection{Computability on Cantor spaces}

Let $A$ be a finite set with $|A|\geq 2$, then $A^{\NN}$ is a Cantor space when endowed with the prodiscrete topology. A set $U \subset A^{\NN} = \{x\colon \NN \to A\}$ is called \define{effectively open} or $\Sigma_1^0$ if there is a Turing machine which enumerates a sequence of words $(w_i)_{i \in \NN}$ in $A^*$ such that $U = \bigcup_{i \in \NN}[w_i]$ where $[w_i] = \{x \in A^{\NN} : w_i \mbox{ is a prefix of } x\}$. A set $C\subset A^{\NN}$ is \define{effectively closed} or $\Pi_1^0$ if it is the complement of an effectively open set. 

Let $A,B$ be alphabets, we say that a map $f\colon X\subset A^{\NN} \to B^{\NN}$ is \define{computable} if there is a Turing machine which on input $w\in B^*$ outputs a sequence of words that describes an effectively open set $U_w\subset A^{\NN}$ such that $f^{-1}([w]) = U_w \cap X$. Given $X\subset A^{\NN}$ and $Y\subset B^{\NN}$, a bijective map $f\colon X \to Y$ is called a recursive homeomorphism if both $f$ and its inverse are computable.

The main goal of this article is to deal with computability properties of subshifts on groups, which can be seen abstractly as subsets of a Cantor space. This identification is straightforward when $G$ has decidable word problem. Indeed, in this case the group $G$ can be computably identified with $\NN$ endowed with a group operation that is computable as a map $\NN^2\to \NN$. In this setting, a group isomorphism $\nu\colon\NN\to G$ is said to be a computable numbering of $G$. Such a numbering allows us to define a homeomorphism $\delta \colon A^\NN\to A^G$ by $\delta(x) = (x_{\nu^{-1}(g)})_{g\in G}$. We declare this homeomorphism to be computable, so we can unambiguously speak about $\Sigma_1^0$ and $\Pi_1^0$ subsets of $A^{G}$. This identification works naturally for finitely generated free groups.

Let us now consider a general group $G$ which is finitely generated by $S$. Let $F(S)$ be the free group generated by $S$, and let $\rho\colon F(S)\to G$ be the canonical surjective group homomorphism. We define an injective and continuous function $\rho^* \colon A^G\to A^{F(S)}$ given by $\rho^*(x)(w) = x(\rho(w))$ for every $w \in F(s)$.

\begin{definition}
    Let $G$ be finitely generated by $S$. For a subshift $X\subset A^G$, we denote by $\widehat{X}\subset A^{F(S)}$ the \textbf{pullback subshift} $\widehat{X} = \rho^*(X)$
\end{definition}

Notice that $\ker(\rho)$ acts trivially on $\widehat{X}$, and thus the shift action $F(S) \curvearrowright \widehat{X}$ can be identified with $G\curvearrowright\widehat{X}$. From a dynamical point of view, the pullback subshift has all the information that we need, as the actions $G\curvearrowright X$ and $G\curvearrowright \widehat{X}$ are topologically conjugate. However, it is much easier to deal with a subshift on the free group from the point-view of recursion theory. For instance, it is clear that a morphism on $A^{F(S)}$ is computable, and thus every subshift which is topologically conjugate to $\widehat{X}$ is recursively homeomorphic to it. In particular, the recursive properties of the pullback do not depend upon the set of generators we choose. Now we can give the general definition of effective subshift.

\begin{definition}
    Let $G$ be a finitely generated group and $S$ a finite set of generators. A subshift $X\subset A^{G}$ is called \define{effective} if the pullback subshift $\widehat{X}$ is a $\Pi_1^0$ subset of $A^{F(S)}$.
\end{definition}

In the case where $G$ is also recursively presented, an equivalent and more intuitive way of thinking about an effective subshift $X\subset A^G$ is the following: a pattern coding is a map $c \colon F \to A$ where $F$ is a finite subset of $S^*$. Given a pattern coding, its cylinder set is given by $[c]= \{x \in A^G : \mbox{ for every } w \in F, x(\underline{w})=c(w)\}$. In a recursively presented group, an effective subshift is one for which there is a Turing machine which enumerates a sequence of pattern codings $(c_i)_{i \in \NN}$ such that \[X = A^G \setminus \bigcup_{i \in \NN, g \in G}g[c_i].\]

For non-recursively presented groups, a subshift which satisfies this property is called \define{effectively closed by patterns} and it is a strictly weaker notion than being effective. An interested reader can find more about the relation between these two notions in~\cite{barbieri_carrasco_rojas_2024_effective}.


\section{Medvedev degrees and basic properties}\label{section:Medvedev}

    \subsection{The lattice of Medvedev degrees}
    Here we provide a brief review of the lattice of Medvedev degrees $\M$. These degrees were introduced in \cite{medvedev_Degrees_1955} with the purpose of relating propositional formulas with mass problems. Important sources are the surveys \cite{hinman_survey_2012}, \cite{sorbi_Medvedev_1996}, see also \cite{lewis_Topological_2011}. 

    Intuitively, a mathematical problem $P$ has a higher Medvedev degree than a mathematical problem $Q$ if every solution to the problem $P$ can be used to compute a solution to the problem $Q$. This intuition can be made precise by defining a pre-order relation $\leq$ on subsets of $\{0,1\}^\NN$,  where each set is interpreted as the set of solutions of a fixed mathematical problem. Given two sets $P,Q\subset \{0,1\}^\NN$, we say that $P$ is \define{Medvedev reducible} to $Q$ when there is a partial computable function $\Psi$ on $\{0,1\}^\NN$, defined on $Q$ and such that $\Psi(Q)\subset P$. We abbreviate this relation by $P\leq Q$. We say that $P$ is \define{Medvedev equivalent} to $Q$ if both $P\leq Q$ and $Q\leq P$. The set $\M$ of Medvedev degrees is the set of equivalence classes of $\{0,1\}^\NN$ modulo Medvedev equivalence. Notice that Medvedev reduction induces a partial order on $\M$, and we will use the same symbol $\leq$ to compare Medvedev degrees. We denote the Medvedev degree of a set $P$ by $m(P)$. 

    The partially ordered set $(\M,\leq)$  is indeed a distributive lattice: there is a minimal element denoted $0_\M$, called the \textit{trivial} degree, an operation $\wedge$ of infimum, an operation $\vee$ of supremum, and a maximal element that here will be denoted $1_\M$. Given two sets $P$ and $Q$, we have the following:
    \begin{itemize}
        \item $m(P)=0_\M$ when $P$ has a computable element. In intuitive terms, a mathematical problem is easy with this complexity measure when it has at least one computable solution.
        \item $m(P)=1_\M$ when $P$ is empty. In intuitive terms, a mathematical problem has maximal complexity with this complexity measure when it has no solution.
        \item $m(P)\vee m(Q)$ equals the degree of the set $\{x\in\{0,1\}^\NN : (x_{2n})_{n\in\NN}\in P\text{ and } (x_{2n+1})_{n\in\NN}\in Q \}$. In intuitive terms, $m(P)\vee m(Q)$ is the difficulty of solving both problems $P$ and $Q$ simultaneously. 
        \item $m(P)\wedge m(Q)$ equals the degree of the set $\{0x :x\in P \}\cup \{1x :x\in Q \}$. Here $0x$ stands for the concatenation of the finite word $0$ with the infinite sequence $x$.   In intuitive terms, $m(P)\wedge m(Q)$ is the difficulty of solving at least one of the problems $P$ and $Q$. 
    \end{itemize}

    A sublattice of $(\M,\leq)$ that will be relevant for us is that of $\Pi_1^0$ degrees. These are the degrees of nonempty $\Pi_1^0$ subsets of $\{0,1\}^\NN$. This lattice contains the minimal element $0_{\M}$, but also contains a maximal element which corresponds to the degree of the class of all PA-complete sets. See~\cite{Cenzer_1999_survey_in_Pi_10_degrees} for a survey on these results.

    \subsection{Medvedev degrees as a dynamical invariant for shift spaces}

    The goal of this section is to define Medvedev degrees of subshifts of $A^G$, where $G$ is an arbitrary finitely generated group. Therefore we shall need to assign Medvedev degrees to sets not contained in $\{0,1\}^\NN$. This is relatively straightforward when $G$ has decidable word problem. Indeed, then $G$ admits a computable bijection $\nu\colon \NN\to G$, and this bijection yields a computable homeomorphism \[\delta \colon A^\NN\to A^G,x\mapsto  (x_{\nu^{-1}(g)})_{g\in G}.\]
    
    With this homeomorphism, we can define the Medvedev degree of a set $X\subset A^G$ by setting $m(X)=m(\delta^{-1}(X))$. This degree does not depend on the choice of $\nu$, see  \cite{carrasco-vargas_geometric_2023_} for details. Having defined Medvedev degrees of subshifts on groups with decidable word problem, we can extend this to the general case as follows. 
    
    \begin{definition}
        Let $G$ be a finitely generated group and let $X\subset A^G$ be a subshift. Let $\widehat{X}$ be the pullback subshift of $X$ to a free group cover of $G$. We define the Medvedev degree $m(X)$ as $m(\widehat{X})$.
    \end{definition}
       
    This definition is also independent of the chosen generating set of $G$ as all those pullbacks are recursively homeomorphic, the reader is referred to \cite[Section 3.3]{barbieri_carrasco_rojas_2024_effective} for a proof of this fact. We now prove some simple facts about $m$ as a dynamical invariant for subshifts. Given two subshifts $X,Y$, denote by $X\times Y$ their direct product as dynamical systems and by $X\sqcup Y$ their disjoint union.
    
    \begin{proposition}\label{prop:basic-properties}
        Let $G$ be a finitely generated group and let $X$, $Y$ be $G$-subshifts. 
        \begin{enumerate}
            \item If there is a topological morphism from $X$ to $Y$, then $m(X)\geq m(Y)$. In particular, the Medvedev degree of a subshift does not increase under factors and is a topological conjugacy invariant. 
            \item $m(X\times Y)=m(X)\vee m(Y)$.
            \item $m(X\sqcup Y)=m(X) \wedge m(Y)$.
        \end{enumerate}
    \end{proposition}
    \begin{proof}
        Let $S$ be a finite generating set for $G$, and let $\widehat X\subset A^{F(S)}$, $\widehat Y\subset B^{F(S)}$ be the associated pullback subshifts in $F(S)$. For the first item note that if there is a topological morphism from $X$ to $Y$, then there is also a topological morphism from $\widehat X$ to $\widehat Y$ as subshifts on $F(S)$. By the Curtis-Hedlund-Lyndon theorem, there is a sliding block code $\phi\colon A^{F(S)}\to B^{F(S)}$ whose restriction to $\widehat X$ is a topological morphism to $\widehat Y$. This sliding block code is clearly a computable function, so we have $m(\widehat X)\geq m(\widehat Y)$. It follows that $m(X)\geq m(Y)$ by definition.  

        The above argument shows that the Medvedev degree of a subshift is a conjugacy invariant, thus we may identify $X\times Y$ with the subshift \[Z = \{(x(g),y(g))_{g\in G} : x,y\in X,Y\}.\]
        By our discussion above, $\widehat X$ and $\widehat Y$ are recursively homeomorphic to sets $P,Q\subset\{0,1\}^\NN$. Using this fact, it is straightforward that $\widehat  Z$ is recursively homeomorphic to $\{x\in\{0,1\}^\NN : (x_{2n})_{n\in\NN}\in P\text{ and } (x_{2n+1})_{n\in\NN}\in Q \}$. As the Medvedev degree of this set equals $m(P)\vee m(Q)$, it follows that $m(Z)=m(X)\vee m(Y)$.

        Again, by taking a topologically conjugate version of $X\sqcup Y$, we can assume that $A$ and $B$ are disjoint alphabets, and that $Z=X\cup Y$ is a subshift on alphabet $A\cup B$. By our discussion above, $\widehat X$ and $\widehat Y$ are recursively homeomorphic to sets $P,Q\subset\{0,1\}^\NN$. Using this fact, it is  straightforward that $\widehat Z$ is recursively homeomorphic to  $\{0x :x\in P \}\cup \{1x :x\in Q \}$. As the Medvedev degree of this set equals $m(P)\wedge m(Q)$, it follows that $m(Z)=m(X)\wedge m(Y)$. 
    \end{proof}

    \begin{definition}
        Let $G$ be a finitely generated group. We denote \begin{enumerate}
            \item $\msft{G} = \{ m(X) : X \mbox{ is a nonempty $G$-SFT}\}.$
            \item $\msof{G} = \{ m(X) : X \mbox{ is a nonempty sofic $G$-subshift} \}.$
            \item $\meff{G} = \{ m(X) : X \mbox{ is a nonempty effective $G$-subshift}\}.$
        \end{enumerate}
        \end{definition}
    It is immediate that all three of the above classes are invariants of isomorphism, as subshifts in two isomorphic groups can be lifted through pullbacks to the same free group. Moreover, each of these three sets is a sublattice: it is closed under the operations $\vee$ and $\wedge$ and always contains $0_{\M}$. This follows from \Cref{prop:basic-properties} and the fact that the classes of SFTs, sofic, and effective subshifts are closed by direct products and disjoint unions.

    Next we will state the basic relations between these three sets. For this, we use that effective subshifts are $\Pi_1^0$ sets, and the following fact that follows from~\cite[Corollary 2.8]{aubrun_notion_2017} and~\cite[Corollary 7.7]{barbieri_carrasco_rojas_2024_effective}.
    
    \begin{proposition}\label{prop:sft-and-sofic-subshifts-are-effective}
            Let $G$ be a finitely generated and recursively presented group. Then SFTs and sofic subshifts on $G$ are effective. \end{proposition}

    It is clear that for any group $G$ we have $\msft{G}\subset \msof{G}$. \Cref{prop:sft-and-sofic-subshifts-are-effective} says that the relation $\msof{G}\subset \meff{G}$ holds for recursively presented groups. We remark that some non-recursively presented groups (such as finitely generated simple groups whose word problem is not in $\Pi_2^0$, see~\cite[Proposition 2.10]{Barbieri_Sablik_Salo_2021}) only admit the trivial effective action and thus $\meff{G}= \{ 0_{\M}\}$, thus the inclusion $\msof{G}\subset \meff{G}$ does not necessarily hold.

    \begin{observation}\label{prop:msft-and-msoff-are-contained-in-Pi-1-degrees}
        Let $G$ be a recursively presented group. Then all three classes $\msft{G}$, $\msof{G}$ and $\meff{G}$ are contained in the class of $\Pi_1^0$ degrees. 
    \end{observation}

        
\section{Transference results}\label{section:transference}

Next we shall study how the space of Medvedev degrees of subshifts on groups behaves with respect to basic relations in group theory.

\subsection{Elementary constructions}

Let $\rho \colon G \to H$ be a group homomorphism. Then $\rho$ induces a map $\rho^* \colon A^H \to A^G$ through \[ \rho^*(x)(g) = x(\rho(g)) \mbox{ for every } x\in A^H, g\in G. \]
It is clear from the above definition that for any subshift $X\subset A^H$, then $\rho^*(X)\subset A^G$ is also a subshift.

    \begin{lemma}\label{lema:medvedev_degree_of_pullback}
        For any finitely generated group $G$, epimorphism $\rho \colon G \to H$ and subshift $X\subset A^H$, we have \[ m(X) = m(\rho^*(X)).  \]
    \end{lemma}

    \begin{proof}
        Note first that as $G$ is finitely generated and $\rho$ is an epimorphism, then $H$ is also finitely generated. Fix finite symmetric generating sets $S,U$ for $G$ and $H$ respectively. Let $\widehat{X}$ be the pullback of $X$ to $F(U)$. For each $s \in S$, let $\bar{\rho}(s) \in U^*$ be a word which represents $\rho(s)$ in $H$. For $\widehat{x} \in \widehat{X}$, we can define $\widehat{y}\colon F(S) \to A$ by \[\widehat{y}(s_1\dots s_n) = \widehat{x}(\bar{\rho}(s_1) \dots \bar{\rho}(s_n)).\]
        It is clear that $\widehat{y} \in \widehat{\rho^*(X)}$ and that the map which sends $\widehat{x}$ to $\widehat{y}$ is computable. It follows that $m(\rho^*(X)) \leq m(X)$. 

        Conversely, for every $u \in U$, fix $w_u \in S^*$ such that $\rho(\underline{w_u}) = u$. Given as input $\widehat{y} \in \widehat{\rho^*(X)}$ we define $\widehat{x}\colon F(U)\to A$ by \[\widehat{x}(u_1\dots u_n) = \widehat{y}(w_{u_1}\dots w_{u_n}).  \]
        This map again is clearly computable and $x \in \widehat{X}$. It follows that $m(X) \leq m(\rho^*(X))$.
    \end{proof}

    \begin{definition}
        Let $H \leqslant G$ be a subgroup and $X\subset A^H$ a subshift. The \define{free extension} of $X$ to $G$ is the subshift \[  \widetilde{X} = \{ x \in A^{G} : \mbox{ for every } g \in G, (x(gh))_{h \in H} \in X\}.   \]
    \end{definition}

    It is clear from the definition that any set of forbidden patterns that defines $X$ in $H$ also defines $\widetilde{X}$ in $G$. In particular it follows that if $X$ is an SFT (respectively sofic, effectively closed by patterns) then so is $\widetilde{X}$.
    
\begin{lemma}\label{free-extensions}
    Let $H\leqslant G$ be finitely generated groups and let $X\subset A^H$ be a subshift. Let $\widetilde{X}$ be the free extension of $X$ to $G$. Then $m(X) \leq m(\widetilde{X})$. Furthermore, if either
    \begin{enumerate}
        \item[(1)] $G$ is recursively presented and $H$ has decidable membership problem in $G$,
        \item[(2)] $H$ has finite index in $G$,
    \end{enumerate}then $m(X) = m(\widetilde{X})$. 
\end{lemma}

\begin{proof}
   Let $S,U\subset G$ be two finite symmetric sets with $S\subset U$ such that $S$ generates $H$ and $U$ generates $G$. As $F(S)\subset F(U)$, the map which on input $\widetilde{x} \in \widehat{\widetilde{X}}$ returns its restriction to $F(S)$ is computable and yields an element $x\in \widehat{X}$. Thus $m(X) \leq m(\widetilde{X})$.

    Conversely, suppose (1) holds. Consider the algorithm which on input $x\colon F(S)\to A$ lists all elements of $F(U)$ in lexicographical order $(u_i)_{i \in \NN}$. When it lists $u_n$, it checks in increasing order for every $k < n$ (using the algorithm for the membership problem of $H$ in $G$) whether $\underline{u_k^{-1}u_n} \in H$. If it is the case for some $k$, then it computes (using that $G$ is recursively presented) a word $w_n \in F(S)$ such that $\underline{u_k^{-1}u_nw_n^{-1}} = 1_G$ and sets $\widetilde{x}(u_n) = x(w_n)$; Otherwise, if $\underline{u_k^{-1}u_n} \notin H$ for every $k < n$, it sets $\widetilde{x}(u_n)=x(1_{F(S)})$. This yields a map $\widetilde{x} \colon F(U) \to A$ which satisfies that $\widetilde{x} \in \widehat{\widetilde{X}}$ if and only if $x \in \widehat{X}$, thus $m(\widetilde{X}) \leq m(X)$.

    Finally, suppose (2) holds. Let $T$ be a finite set such that $TH=G$. Without loss of generality, we may suppose $T\subset U$ and that each $u \in U$ is written as $u = t_uw_u$ with $t_u \in T$ and $w_u \in S^*$. Given such a set $U$, there is a computable map which on input $v \in U^*$ returns $t_v \in T$ and $u_v \in S^*$ such that $\underline{v} = \underline{t_v u_v}$ (to do this, we only need to store the values of all multiplications $tk$ and $s t$ for $s \in S, t,k \in T$, of which there are finitely many). Using this map, it follows that the function which on input $x\colon F(S)\to A$ returns $\widetilde{x} \colon F(U) \to A$ given by $\widetilde{x}(v)=x(u_v)$ is computable and $\widetilde{x} \in \widehat{\widetilde{X}}$ if and only if $x \in \widehat{X}$. It follows that $m(\widetilde{X}) \leq m(X)$.
\end{proof}

\begin{corollary}
    Let $H \leqslant G$ be finitely generated groups. If there exists an $H$-SFT with nontrivial Medvedev degree, then there exists a $G$-SFT with nontrivial Medvedev degree.
\end{corollary}

\begin{corollary}\label{cor:subgroups_decidable_membership}
    Let $H\leqslant G$ be finitely generated, recursively presented groups and suppose $H$ has decidable membership problem in $G$, then $\msft{H} \subset \msft{G}$ and $\msof{H} \subset \msof{G}$. 
\end{corollary}

\subsection{Commensurability and quotients}

    Here we prove that $\msft{G}$ and $\msof{G}$ are commensurability invariants. For this we apply a construction of Carroll and Penland~\cite{carroll_Periodic_2015}. Let $H\leqslant G$ be finitely generated groups such that $H$ has finite index in $G$, and let $T=\{t_1=1_G,t_2,\dots,t_r\}\subset G$ be a finite set such that $HT=G$. Given an alphabet $A$, we let $B=A^T$ and $\Psi\colon A^G\to B^H$ be the map given by $\Psi(x)(h)=(h^{-1}x)|_T $. The image of a $G$-subshift $X$ by $\Psi$ is an $H$-subshift. Moreover, $X$ is a $G$-SFT (respectively sofic subshift) if and only if $\Psi(X)$ is an $H$-SFT (respectively sofic $H$-subshift), see~\cite[Section 3.1]{carroll_Periodic_2015} for the proof.

    \begin{lemma}\label{higher-block-preserves-m}
        Given a subshift $X\subset A^G$, $m(X)=m(\Psi(X))$.
    \end{lemma}
    \begin{proof}
        Let $U,S \subset G$ be symmetric finite generating sets with $U\subset S$ for $H$ and $G$ respectively. Fix a set of elements $\{h_1,\dots,h_r\}$ in $F(S)$ such that $h_i$ corresponds to $t_i$ in $G$.
        
        For the inequality $m(X)\geq m(\Psi(X))$, we define a computable function $\widehat \Psi\colon A^{F(S)}\to B^{F(U)}$ that represents $\Psi$ on the pullback subshifts. Given  $x\in A^{F(S)}$, $\widehat \Psi (x)\colon F(U)\to B$ is defined by $\widehat{\Psi}(x)(w)(t_i) = x(w h_i)$ for $w\in F(U)$ and $t_i \in T$. It is clear that $\widehat \Psi$ is computable and that every element in $\widehat X$ is mapped by $\widehat \Psi$ to $\widehat{\Psi(X)}$, thus we obtain $m(X)\geq m(\Psi(X))$. 

        For the remaining inequality, we proceed analogously as in~\Cref{free-extensions}. There is an algorithm which on input $v \in S^*$, returns $t_v \in T$ and $u_v\in U^*$ such that $\underline{v} = \underline{u_v t_v}$. Consider the map which on input an element $x\in \widehat{\Psi(X)}$, returns $y\in \widehat X$ defined as follows: for $v\in F(S)$ set $y(v)=x(u_v)(t_v)$. This defines a computable function from $\widehat {\Psi(X)}$ to $\widehat X$, and proves the desired inequality. 
    \end{proof}

Recall that two groups $G$ and $H$ are said to be \define{commensurable} when there exist finite index subgroups $G' \leqslant G$ and $H' \leqslant H$ which are isomorphic. 

    \begin{proposition}\label{prop:commensurable_same_degrees}
        Let $G$ and $H$ be two commensurable finitely generated groups. Then $\msft{G} = \msft{H}$ and $\msof{G} = \msof{H}$.
    \end{proposition}
    \begin{proof}
        As the set of Medvedev degrees is an isomorphism invariant of groups, we assume without loss of generality that $H$ is a finite index subgroup of $G$.
        
        Let $X$ be an SFT on $H$. As $H$ has finite index, it follows by~\Cref{free-extensions} that the free extension of $X$ in $G$ is a $G$-SFT that has the same Medvedev degree as $X$. Therefore $ \msft{H} \subset \msft{G}$. 
        
        Conversely, given a $G$-SFT $X$, we have that $\Psi(X)$ is an $H$-SFT. Moreover, it has the same Medvedev degree as $X$ by~\Cref{higher-block-preserves-m}. This shows that $ \msft{G} \subset \msft{H}$. The proof that $\msof{G} = \msof{H}$ is identical. 
    \end{proof}

    \begin{proposition}\label{prop:quotients_fg}
        Consider a short exact sequence of groups $1 \to N \to G \to H \to 1$. If both $G$ and $N$ are finitely generated then \[ \msft{H} \subset \msft{G} \mbox{ and } \msof{H} \subset \msof{G}.   \]
    \end{proposition}

    \begin{proof}
        Denote by $\rho\colon G\to H$ the epimorphism in the short exact sequence and let $X\subset A^H$ be an SFT. By~\Cref{lema:medvedev_degree_of_pullback}, we have that the pullback $\rho^*(X)$ satisfies $m(\rho^*(X))=m(X)$. Therefore in order to show that $\msft{H} \subset \msft{G}$ it suffices to show that $\rho^*(X)$ is an SFT.

        As $X$ is an SFT, there exists a finite set $F\subset H$ and $L \subset A^F$ such that $x \in X$ if and only if $(hx)|_{F} \in L$ for every $h \in H$. For every $f \in F$, choose $g_f \in G$ such that $\rho(g_f) = f$ and let $U =\{g_f : f \in F \}\subset G$.
        Let $S\subset G$ be a finite symmetric generating set of $N$ which contains the identity and consider the set $E = U \cup S$. Let \[ W = \{ w \in A^{E} : (w(g_f))_{f \in F} \in L \mbox{ and } w(s) = w(1_G) \mbox{ for every } s \in S.\}.   \]

        It is immediate from the definition that $\rho^*(X) = \{y \in A^G : (gy)|_{E} \in W \mbox{ for every } g \in G\}$. In particular, this means that $\rho^*(X)$ is an SFT (a set of defining forbidden patterns is $A^{E}\setminus W$.

        Now let $Y\subset B^{H}$ be a sofic shift and let $X$ be some SFT such that $\phi \colon X \to Y$ is a topological factor map given by some local map $\Phi \colon A^{K}\to B$ for some finite $K\subset H$. As before, for every $k \in K$ choose $g_k \in G$ such that $\rho(g_k) = k$ and let $V =\{g_k : k \in K \}\subset G$. By the previous argument, we have that $\rho^*(X)$ is also an SFT. Let $\phi^*\colon \rho^*(X)\to B^{G}$ be given by the local rule $\Phi^* \colon A^V \to B$ defined by \[  \Phi^*( (a_v)_{v \in V}) = \Phi( (a_{g_k})_{k \in K}). \]

        It is clear from the definition that $\phi^*(\rho^*(X)) = \rho^*(Y)$, thus $\rho^*(Y)$ is also a sofic shift and thus $\msof{H} \subset \msof{G}$.
    \end{proof}

    \begin{remark}
        Without the assumption that $N$ is finitely generated,~\Cref{prop:quotients_fg} does not hold. For instance, for the short exact sequence $1 \to [F_2,F_2] \to F_2 \to \ZZ^2 \to 1$ we have that $\msft{F_2} = \{0_\M\}$ (\Cref{prop:virtually_free_have_0_degree}) but $\msft{\ZZ^2}$ is the set of all $\Pi_1^0$ Medvedev degrees. We note that this same example has been used to exhibit that the same assumptions are needed to transfer strong aperiodicity~\cite{jeandel_Translationlike_2015} and the domino problem~\cite{aubrun_Domino_2018}.
    \end{remark}

    \subsection{Bounded actions and the orbit membership problem}
    
    Let $G,H$ be two finitely generated groups, endowed with word metrics. A right action $\ast$ of $H$ on $G$ is called \define{bounded} if for every $h\in H$, the map given by $g\mapsto g\ast h$ is at bounded distance from the identity function on $G$. An equivalent condition is that, for some (equivalently, every) finite generating set $S\subset H$, there exists a finite set $F\subset G$ with $g^{-1}(g\ast s)\in F$ for every $g\in G$ and $s\in S$. A right action is called \define{translation-like} if it is bounded and free. 

    The set of all bounded actions of $H$ on $G$ with fixed parameters $S$ and $F$ can be parametrized by a $G$-subshift in which every symbol is an element of $F^S$ which encodes locally the action by every generator. Moreover, we can overlay this subshift with configurations coming from an $H$-subshift and use it to produce a $G$-subshift which inherits some of the properties of the overlayed $H$-subshift. This technique was introduced in \cite{jeandel_Translationlike_2015}, and has been used several times to construct subshifts on groups with interesting properties~\cite{barbieri_entropies_2021, carrasco-vargas_geometric_2023_, cohen_Strongly_2021, barbieri_geometric_1}. In what follows we will briefly describe the construction, and refer the reader to~\cite{jeandel_Translationlike_2015} for details.
    
    Let $S$ be a finite and symmetric generating set for $H$, and $F\subset G$ a finite set. Let $B$ be the set of maps from $S$ to $F$. Given $x\in B^G$, a group element $g\in G$, and a word $w\in S^\ast$, we denote by $\Phi(g,x,w)$ the group element in $G$ obtained by interpreting $x(g)$ as an arrow from $g$ to $gx(g)$ labeled by $S$, and following these arrows as indicated by the word $w$.  More precisely, we define $\Phi(g,x,\epsilon)=g$, and then for $s\in S$ and $w\in S^\ast$ we set $\Phi(g,x,ws)=\Phi(g,x,w)\cdot x(\Phi(g,x,w))(s)$.  
    \begin{definition}
        Given $H,G,S,F$ and $B$ as above, the \define{subshift of bounded actions} $\operatorname{T} \subset B^G$ is the set of all configurations $x\in B^G$ which satisfy that $\Phi(g,x,w)=g$, for every $g\in G$ and $w\in S^\ast$ that represents the identity element in $H$.  
    \end{definition}
    We note that in general, the $G$-subshift $\operatorname{T}$ is not of finite type. However, in the case where $G$ is finitely presented the condition on configurations of $\operatorname{T}$ can be enforced with finitely many relations which in turn can be encoded through finitely many forbidden patterns, thus in this case $\operatorname{T}$ is a $G$-SFT.
    
    Observe that for every $x\in B^G$, the free monoid $S^*$ acts (on the right) on $G$ by $g \cdot w = \Phi(g,x,w)$ for $w \in S^*$. From our definition, it is clear that $\operatorname{T}$ is precisely the subset of configurations which induce an action of $H$. With this in mind, we will extend the notation above so that given $g\in G$, $h\in H$, and $x\in \operatorname{T}$, we write $\Phi(g,x,h)$ to denote $\Phi(g,x,w)$ for any $w \in S^*$ which represents $h$. 

    \begin{definition}
        Let $X\subset A^H$ be a subshift. We define $\operatorname{T}[X]\subset(A\times B)^G$ as the set of all configurations $(x,y)\in A^G \times B^G$ such that $y\in T$ and for every $g\in G$, the configuration $\left(x(\Phi(g,y,h))\right)_{h \in H}$ lies in $X$. 
    \end{definition}

    \begin{proposition}[\cite{jeandel_Translationlike_2015}, Section 2]\label{T[X]-es-SFT}
        The sets $\operatorname{T}$ and $\operatorname{T}[X]$ are $G$-subshifts. Moreover, if $H$ is finitely presented and $X$ is an $H$-SFT, then $\operatorname{T}[X]$ is a $G$-SFT. 
    \end{proposition} 

    Notice that if $Y$ is a sofic $H$-subshift and $H$ is finitely presented we have that $\operatorname{T}[Y]$ is a sofic $G$-subshift. Indeed, it suffices to take an $H$-SFT $X$ for which there exists a topological factor map $\phi \colon X \to Y$, let $W = \{ (x,y) \in X\times Y : \phi(x) = y\}$ and notice that $W$ is an $H$-SFT. Then $T[W]$ is also a $G$-SFT and $T[Y]$ can be obtained as a projection of $T[W]$.
    
    Next we will study the relation between the degrees $m(X)$ and $m(\operatorname{T}[X])$. In order to do that, we need to introduce the following notion (see also \cite{bogopolski2010orbit,carrasco-vargas_geometric_2023_}).

    \begin{definition}
        Let $G,H$ be finitely generated groups and $R$ a finite set of generators for $G$. A right action $\ast$ of $H$ on $G$ has \define{decidable orbit membership problem} if there is a Turing machine which on input $u,v \in R^*$ decides whether $\underline{u}$ and $\underline{v}$ lie in the same $H$-orbit.
    \end{definition}
    Observe that when $G$ is recursively presented and the action is computable, this notion is equivalent to the existence of a decidable set $I$ and computable sequence $(u_i)_{i \in I}$ with $u_i \in R^*$ such that the corresponding sequence of elements $(\underline{u_i})_{i \in I}$ in $G$ form a collection of representatives for the orbits of $H$. 

    \begin{proposition}\label{prop:medvedev-degrees-and-translation-like-actions}
        Let $G,H$ be finitely generated groups and let $\operatorname{T}$ be the subshift of bounded actions for parameters $S,F$. Then:
        
        \begin{enumerate}
            \item For every subshift $X\subset A^H$ we have $m(\operatorname{T}[X])\geq m(\operatorname{T})\vee m(X)$.
            \item Suppose that $G$ is recursively presented and that there is a translation-like action $H\curvearrowright G$ which is computable and has decidable orbit membership problem (for the parameters $S$ and $F$ already fixed). For every subshift $X\subset A^H$ we have $ m(\operatorname{T}[X]) = m(X)$.      
        \end{enumerate}
    \end{proposition}

    \begin{proof}
        First observe that it is possible for $\operatorname{T}[X]$ to be empty, but then the inequality holds for the trivial reason that $m(\varnothing)=1_{\M}$ is maximal. Suppose now that $m(\operatorname{T}[X])$ is nonempty. The projection to the second coordinate of the alphabet $A\times B$ yields a function $m(\widehat{\operatorname{T}[X]})\to \widehat{\operatorname{T}}$, which is computable by~\Cref{prop:basic-properties}. It follows that $m(\operatorname{T}[X])\geq m(\operatorname{T})$. We now verify the nontrivial inequality $m(\operatorname{T}[X])\geq m(X)$. 
        
        Let $R$ be a finite generating set for $G$ which contains the range of the alphabet $B$ of $\operatorname{T}$, and let $F(R)$ be the free group on $R$. Given $x\in B^{F(R)}$, $u\in F(R)$, and $w\in S^\ast$, we denote by $\widehat{\Phi}(u,x,w)$ the element in $F(R)$ defined as follows. Recursively, $\widehat{\Phi}(u,x,\epsilon)=u$, and for $s\in S$ and $w\in S^\ast$, $\Phi(u,x,ws)=\Phi(u,x,w)\cdot x(\Phi(u,x,w))(s)$. Then it is clear that $\widehat{\Phi}$ is a computable function. Also note that $\widehat{\Phi}$ and $\Phi$ are compatible, in the sense that for all $x\in \operatorname{T}$ with pullback $\widehat{x}\in \widehat{\operatorname{T}}\subset B^{F(R)}$, for all $g\in G$, $u\in F(R)$ with $\underline{u}=g$, and for all $w\in S^\ast$, we have that $\Phi(g,x,w)=\underline{\widehat \Phi (u,\widehat x,w)}$.  We now define 
        
        \[\Psi\colon (A\times B)^{F(R)}\to A^{F(S)}, \mbox{ given by } \Psi(x,y) = \Bigl(u\mapsto \bigl(x(\widehat{\Phi}(1,y,u))\bigr)\Bigr).\]
        
        This function is computable because $\widehat{\Phi}$ is computable. Moreover, it follows from the compatibility of $\widehat{\Phi}$ and $\Phi$,  that $\Psi(\widehat{\operatorname{T}[X]})\subset \widehat{X}$. This shows that $m(\operatorname{T}[X])\geq m(X)$.
        
        We now consider the second item. Suppose there exists a translation-like action $\ast$ as in the statement. We define a configuration $y \in B^{F(R)} $ by $y(u)(s)=(\underline{u}^{-1})(\underline{u}\ast s)$, where $\underline{u}$ is the element of $G$ represented by $u \in F(R)$. As $\ast$ is a computable translation-like action and $G$ is recursively presented, it follows that $y$ is a computable element in $B^{F(R)}$. Moreover, it is clear by its definition that $y \in \widehat{T}$ thus it follows that $m(\operatorname{T})=0_{\M}$. By the first item in the statement, we obtain that $m(\operatorname{T}[X])\geq m(X)$. 
        
        Finally, we prove that $m(X)\geq m(\operatorname{T}[X])$. As the translation-like action $\ast$ has decidable orbit membership problem, we can compute a set of elements $(u_i)_{i\in I}\subset F(R)$, where $I\subset \NN$ is a decidable set, such that the corresponding elements $(\underline{u_i})_{i \in I}$ in $G$ are a collection of representatives for orbits of the action. Let $\Theta\colon \widehat{X} \to \widehat{\operatorname{T}[X]}$ be the map given by $\Theta (z)=(x,y)$, where $y$ is the computable point already defined above, and $x$ is defined by 
        \[
        x(u)=z(w) \mbox{ for } u \in F(R), \mbox{ where } w \in F(S) \mbox{ is such that } \underline{\widehat{\Phi}(u_i,y,w)}=\underline{u} \mbox{ for some } i\in I.
        \]
        The function $\Theta$ is well defined by the above construction and is computable. Indeed, as $G$ is recursively presented, there is an algorithm which, given $u\in F(R)$, computes $i\in I$ and $w\in F(S)$ such that $\underline{u} =\underline{\widehat \Phi (u_i,y, w)}$. It follows that $m(\operatorname{T}[X]) = m(X)$.\end{proof}

    \begin{corollary}\label{prop:translation-like-actions-and-medvedev-degrees}
        Let $G,H$ be finitely generated groups, where $H$ is finitely presented and admits a translation-like action on $G$.
        \begin{enumerate}
            \item For every sofic $H$-subshift $Y$, there is a $G$-SFT $X$ with $m(X)\geq m(Y)$. In particular, if $H$ admits a sofic subshift with nonzero Medvedev degree, then so does $G$.
            \item If $G$ is recursively presented and there is a translation-like action of $H$ on $G$ that is computable and has decidable orbit membership problem, then $\msft{H}\subset\msft{G}$ and $\msof{H}\subset\msof{G}$. 
        \end{enumerate}
    \end{corollary}
    \begin{proof}
    Consider a sofic $H$-subshift $Y$. Let X be an $H$-SFT extension. By~\Cref{prop:basic-properties} we have $m(X)\geq m(Y)$. Let $\operatorname{T}$ be the subshift of bounded actions with parameters such that it is nonempty. By~\Cref{T[X]-es-SFT} we have that $Z= \operatorname{T}[X]$ is a $G$-SFT and by~\Cref{prop:medvedev-degrees-and-translation-like-actions} it follows that $m(Z)\geq m(X)$. We conclude that $m(Z) \geq m(Y)$. 
    
    With the extra hypotheses of the second item of~\Cref{prop:medvedev-degrees-and-translation-like-actions} we have that $m(T[X]) = m(X)$ for any subshift $X$. If we take $X$ an $H$-SFT (respectively a sofic $H$-subshift) then $\operatorname{T}[X]$ is a $G$-SFT (respectively sofic $G$-subshift) and thus we obtain $\msft{H}\subset\msft{G}$ and $\msof{H}\subset\msof{G}$. 
    \end{proof}

    \begin{observation}
        The previous results can be sometimes used as obstructions for the existence of translation-like actions. For instance, the fact that $\ZZ^2$ admits SFTs with nonzero Medvedev degree, and that $F_2$ does not, implies by the previous results that $\ZZ^2$ does not act translation-like on $F_2$. While this fact is rather intuitive, a direct proof of this fact is not that straightforward, see \cite[Section 3]{cohen_counterexample_2019}. 
    \end{observation}

    In these results we have shown how to use the subshift of bounded actions to obtain SFTs with certain Medvedev degrees. We now consider the case of effective subshifts. 

    \begin{proposition}\label{prop:translation-like-action-subshift-is-effective}
        Let $G$ and $H$ be finitely generated  and recursively presented groups. Suppose further that there is a translation-like action $\ast$ of $H$ on $G$. Then:
        \begin{enumerate}
            \item The subshift $\operatorname{T}\subset B^G$ is effective. 
            \item If $X\subset A^H$ is effective, then $\operatorname{T}[X]\subset (A\times B)^G$ is effective. 
        \end{enumerate}
    \end{proposition}
    \begin{proof}
        Let $S$ be a finite and symmetric finite generating set for $H$ and let $F\subset G$ be a finite set such that $g^{-1}(g\ast s)\in F$ for every $s \in S$, $g \in G$. Let $R$ be a finite symmetric generating set for $G$ that contains $F$, and let $\widehat \Phi$ be the function defined in the proof of~\Cref{prop:medvedev-degrees-and-translation-like-actions}. Let $X_0$ be the set of all elements $x\in B^{F(R)}$ such that, for all $g\in F(R)$, and for all $w\in S^\ast$ with $\underline w=1_G$, we have that $\widehat \Phi (g,x,w)$ and $g$ coincide in $F(R)$. As the set $\{w\in S^\ast : \underline{w}=1_G\}$ is recursively enumerable, it follows that $X_0$ is a $\Pi_1^0$ subset of $A^{F(R)}$. Also note that by~\Cref{prop:sft-and-sofic-subshifts-are-effective} $\widehat{A^G}$ is a $\Pi_1^0$ subset of $A^{F(R)}$. As $\widehat{\operatorname{T}}$ equals the intersection of $X_0$ with $\widehat{A^G}$, it follows that $\widehat{\operatorname{T}}$ is a $\Pi_1^0$ set. Therefore $\operatorname{T}$ is an effective subshift. 

        We now prove that, if $X$ is an effective subshift on $H$, then $\operatorname{T}[X]$ is an effective subshift on $G$. For this purpose, let $\Psi\colon (A\times B)^{F(R)}\to A^{F(S)}$ be the computable function defined in the proof of \Cref{prop:medvedev-degrees-and-translation-like-actions}. We also consider the computable function $\Omega\colon (A\times B)^{F(R)}\to B^{F(R)}$ that removes the $A$ component of the alphabet. It follows from the definitions that 
        \[\widehat{\operatorname{T}[X]}=\Omega^{-1}(\operatorname{T})\cap \Psi^{-1}(X).\]

        As the preimage of an $\Pi_1^0$ set by a computable function is a $\Pi_1^0$ set, this proves that $\widehat{T[X]}$ is $\Pi_1^0$. 
    \end{proof}

    A consequence of these results and a theorem of Seward which ensures the existence of a translation-like action of $\ZZ$ on any infinite and finitely generated group~\cite[Theorem 1.4]{seward_Burnside_2014} is the following.
    \begin{corollary}\label{cor:all_infinite_rp_groups_havenonzero_effective_medvedev_degree}
        Every infinite, finitely generated and recursively presented group admits a (nonempty) effective subshift which achieves the maximal $\Pi_1^0$ Medvedev degree.
    \end{corollary}
    \begin{proof}
        Let $G$ be an infinite, finitely generated and recursively presented group, and let $H=\ZZ$. Let $X$ be a nonempty effective $\ZZ$-subshift with maximal $\Pi_1^0$ Medvedev degree, whose existence was proven by Miller in~\cite{miller_Two_2012}. By a result of Seward~\cite[Theorem 1.4]{seward_Burnside_2014}, $G$ admits a translation-like action by $\ZZ$, so (for appropriate parameters) the subshift $\operatorname{T}[X]$ is nonempty. By~\Cref{prop:translation-like-action-subshift-is-effective}, the subshift $\operatorname{T}[X]$ is effective. Finally, by \Cref{prop:medvedev-degrees-and-translation-like-actions}, we have $m(T[X])\geq m(X)$. As $G$ is recursively presented, it follows that $m(T[X])$ is a $\Pi_1^0$ degree and thus we must have $m(T[X])=m(X)$.
    \end{proof}

    \subsection{Quasi-isometries}

    Let $(X,d_X), (Y,d_Y)$ be two metric spaces. A map $f\colon X\to Y$ is a \define{quasi-isometry} if there exists a constant $C\geq 1$ such that
\begin{enumerate}
    \item $f$ is a \define{quasi-isometric embedding}: for all $x,x'\in X$
    \[\frac{1}{C}d_X(x,x') - C \leq d_Y(f(x), f(x'))\leq C d_X(x,x') + C,\]
    \item $f$ is \define{relatively dense}: for all $y\in Y$ there exits $x\in X$ such that $d_Y(y, f(x))\leq C$.
\end{enumerate}

    Two metric spaces are called \define{quasi-isometric} if there is a quasi-isometry between them. Given a quasi-isometry $f\colon X \to Y$, a map $g\colon Y \to X$ is called a \define{coarse inverse} of $f$ if the composition $g\circ f$ is uniformly close to the identity map $\operatorname{Id}_X$. Coarse inverses always exist, and every coarse inverse is in particular a quasi isometry, thus being quasi-isometric is an equivalence relation (see for instance~\cite{drutu_Geometric_2018}).

    Two finitely generated groups $G, H$ are \define{quasi-isometric} if there exists a quasi-isometry between $(G,d_G)$ and $(H,d_H)$ for some choice of word metrics $d_G,d_H$. It is clear that if a map is a quasi-isometry for fixed word metrics, it is also a quasi-isometry for any other choice of word metric (up to modification of the constant $C$). We also remark that if a quasi-isometry between two finitely generated groups exists, then it is necessarily bounded-to-1.

    A well-known construction of Cohen~\cite{cohen_large_2017} shows that if one is given two finitely generated groups $G,H$ and a positive integer $N$, then one can construct a subshift $\operatorname{QI}$ on $G$ whose elements encode all quasi-isometries $f\colon H \to G$ which satisfy that $|\{f^{-1}(g)\}| \leq N$ for every $g \in G$. In particular, if $H$ is quasi-isometric to $G$, this subshift $\operatorname{QI}$ is nonempty for large enough $N$. Furthermore, if one of the groups (equivalently, both\footnote{As being finitely presented is an invariant of quasi-isometry, see~\cite{drutu_Geometric_2018}}) are finitely presented then $\operatorname{QI}$ turns out to be an SFT. The main observation is that in this case, for any subshift of finite type $X\subset A^H$, we can enrich $\operatorname{QI}$ with the alphabet of $X$ and add extra local rules to obtain a $G$-SFT $\operatorname{QI}[X]$ which satisfies the property that in every configuration the encoded copy of $H$ is overlaid with a configuration $x \in X$. 

    The explicit construction of Cohen involves several details which are too lengthy to include in this article, so we will just provide a quick overview of an equivalent formulation. 
    
    Let $S,T$ be finite symmetric sets of generators which contain the identity for $G$ and $H$ respectively. Denote $d_G$ and $d_H$ the words metrics in $G$ and $H$ with respect to these generators. Let $f\colon H \to G$ be a quasi-isometry with constant $C\in \NN$ and suppose that $\max_{g \in G}|f^{-1}(g)| = N$. Let $I = \{1,\dots,N\}$. 
    
    The subshift $\operatorname{QI}$ is defined over the alphabet $B$ whose symbols have the form $b=\left(b_1, \dots,b_N\right)$ and each $b_i$ is either a map $b_i \colon T \to S^{2C}\times I$ or the symbol $\ast$. For $q\in \operatorname{QI}$, $g \in G$ and $i \in I$, let us denote $q(g,i)$ the $i$-th coordinate of $q(g) \in B$. 

    The intuition is the following: if $q \in \operatorname{QI}$, then for every $g \in G$ the tuple $q(g) = (q(g,1),\dots,q(g,N)) \in B$ will encode the set $f^{-1}(g)$ for some quasi-isometry $f$. As every $g \in G$ can have between $0$ and $N$ preimages, each $q(g,i)$ will either encode one of them (and be a map $q(g,i) \colon T \to S^{2C}\times I$) or encode no preimage ($q(g,i) = \ast$). 

    The maps encode only local information about $h \in f^{-1}(g)$. Namely, for every $t \in T$, $q(g,i)(t)$ encodes $f(h)^{-1}f(ht)$ (which by the upper bound on the quasi-isometry, can be represented by a word in $S$ of length at most $2C$) and the index $j \in I$ such that $q(f(ht),j)$ is encoding $ht$.

    The forbidden patterns that define $\operatorname{QI}$ can be described informally with rules as follows: 
    \begin{enumerate}
        \item  The relative dense condition in $\operatorname{QI}$ is encoded by imposing that in every ball of length $C$ in $G$ there is at least some element which has a non-$\ast$ coordinate.
        \item The directions from non-$\ast$ elements must lead to non-$\ast$ elements. More precisely, if $q(g,i)\neq \ast$, then for every $t \in T$ if $q(g,i)(t)=(w,j)$, then $q(g\underline{w},j)\neq \ast$.  
        \item The lower bound in the quasi-isometry is encoded by asking that for any two elements of $G$ with a non-$\ast$ coordinate at distance at most $2C$, there is a path of non-$\ast$ elements of bounded length linking them.
        \item The relations of $H$ are encoded by enforcing that paths associated to words $T^*$ which represent the identity are cycles. This last rule is the only one that requires the groups to be finitely presented in order to obtain an SFT.
    \end{enumerate} 

    Recall that for a word $w \in S^*$ we denote by $\underline{w}$ its corresponding element of $G$. Let $q \in \operatorname{QI}$, by the first rule of $\operatorname{QI}$, there exists $u \in S^*$ with $|u|\leq C$ and $i \in I$ such that $q(\underline{u},i)\neq \ast$. Let $(u,i)$ be lexicographically minimal among all such pairs. Consider the map $\kappa \colon T^* \to S^* \times I$ defined inductively as follows. Fix $\kappa(\epsilon) = (u,i)$. Suppose for some $w\in T^*$ we have $\kappa(w) = (u',i')$ and $q(\underline{u'},i')\neq \ast$. For every $t \in T$, we compute $q(\underline{u'},i')(t) = (v,j)$ and set $\kappa(wt) = (u'v,j)$. Notice that by the second rule $q(\underline{u'v},j) \neq \ast$. Clearly the map $\kappa$ can be computed from a description of $q$.

    The rules for $\operatorname{QI}$ ensure the following two properties
    \begin{enumerate}
        \item The third rule of $\operatorname{QI}$ ensures that the range of $\kappa$ spans the set of all $(g,i)\in G \times I$ for which $q(g,i) \neq \ast$, namely, \[ \{(g,i)\in G \times I : q(g,i) \neq \ast \} = \{ (\underline{u},i)\in G \times I : \kappa(w) = (u,i) \mbox{ for some } w \in T^*\}.      \]
        \item The fourth rule ensures that the map $\kappa \colon T^* \to S^* \times I$ is well defined on $H$, meaning that if $w,w'$ are two words in $T^*$ which represent the same element in $H$, then $\kappa(w)$ and $\kappa(w')$ represent the same pair of elements in $G \times I$. 
    \end{enumerate}

    Given an alphabet $A$ we can enrich the alphabet $B$ by replacing every non-$\ast$ coordinate by pairs of the form $(a,b_i)$ for some $a \in A$. This way from every configuration $\widetilde{q}$ with this enriched alphabet, we can obtain $\widehat{x}\colon T^*\to A$ by setting for $w \in T^*$, $\widehat{x}(w)$ as the $a$-th coordinate of $\widetilde{q}(\kappa(w))$. Furthermore, by the first property of $\kappa$, the map $\widehat{x}$ induces a configuration $x \in A^H$. 
    
    For a subshift $X\subset A^H$, we denote by $\operatorname{QI}[X]$ the subshift over the enriched alphabet where we forbid the occurrence of forbidden patterns of $X$ in the induced configurations. The fundamental property proven by Cohen in~\cite{cohen_large_2017} is that if the groups are finitely presented and $X$ is an $H$-SFT, then $\operatorname{QI}[X]$ is a $G$-SFT.

    \begin{lemma}\label{lem:QI_medvedev}
        Let $G,H$ be finitely generated groups with $G$ recursively presented. Let $X\subset A^H$ be a subshift. We have \[m(\operatorname{QI}[X]) = m(\operatorname{QI})\vee m(X).\]
    \end{lemma}

    \begin{proof}
        Let $\widehat{q}\in \widehat{\operatorname{QI}[X]}$. It is clear that one can compute an element of $\widehat{\operatorname{QI}}$ from $\widehat{q}$ by just erasing the alphabet coordinate. As the map $\kappa$ defined above is computable from $\widehat{q}$, we may compute $x \in \widehat{X}$ from this map and $\widehat{q}$. Thus we have that $m(\operatorname{QI}[X]) \geq m(\operatorname{QI}) \vee m(X)$.

        Conversely, let $q \in \widehat{\operatorname{QI}}$ and $x \in \widehat{X}$. We can compute $\kappa$ from $q$ as before. For $u \in S^*$ and $i \in I$, if $q(u,i)= \ast$, set $\widehat{q}(u,i) = \ast$. Otherwise, we know that there must exist $v \in T^*$ such that $\kappa(v) = (u',i)$ with $\underline{u}=\underline{u'}$. Thus if we run the algorithm that enumerates the word problem on $u'u^{-1}$ on every pair $(u',i)$ which occurs in the range of $\kappa$ this algorithm finishes. Set $\widehat{q}(u,i) = (x(v),q(u,i))$. With this it is clear that $\widehat{q}\in \operatorname{QI}[X]$ and thus $m(\operatorname{QI})\vee m(X) \geq m(\operatorname{QI}[X])$.
    \end{proof}

    \begin{remark}\label{admitir-sft-incalculable-es-invariante-de-qi}
        It follows from this result that for finitely presented groups, admitting an SFT with nonzero Medvedev degree is an invariant of quasi-isometry. We also mention that in the previous proof, the direction $m(\operatorname{QI}[X]) \geq m(\operatorname{QI})\vee m(X)$ does not require the groups to be recursively presented. 
    \end{remark}

    \begin{proposition}\label{prop:QI-computableinverse}
        Suppose $G,H$ are finitely generated groups with decidable word problem. If there exists a computable quasi-isometry $f\colon H \to G$, then there exists a computable coarse inverse $g\colon G \to H$.
    \end{proposition}

       \begin{proof}
        Fix finite sets of generators $S,T$ which induce metrics $d_G$ and $d_H$ for $G$ and $H$ respectively. Let $f \colon H \to G$ be a computable quasi-isometry with constant $C>0$ and $\widehat{f}\colon T^* \to S^*$ its computable representative. We construct $\widehat{g}\colon S^* \to T^*$ through the following algorithm.

        Given $w \in S^*$, run the algorithm which decides the word problem on $uw^{-1}$ for all words $u$ of length at most $|w|$ and replace $w$ by the lexicographically minimal $u$ which is accepted. In increasing lexicographical order, compute $d_G(u,\widehat{f}(v))$ for $v \in T^*$. Return $\widehat{g}(w) = v$ where $v$ is the first word for which $d(u,\widehat{f}(v)) \leq C$.

        As $f$ is relatively dense, this algorithm is guaranteed to stop on every input, and thus $\widehat{g}$ is a total computable map. Furthermore, the first step ensures that for two words $w,w'$ which represent the same element of $G$, then $\widehat{g}(w) = \widehat{g}(w')$ and thus it represents a map $g \colon G \to H$.

        Finally, for any $t\in G$, we have $d_G(f\circ g(t),t) \leq C$ and thus $d_{\infty}(f\circ g,\operatorname{Id}_G)$ is finite. This implies that $g$ is a coarse inverse of $f$ and thus a quasi-isometry. \end{proof}

    \begin{definition}
        We say that two finitely generated groups are \define{computably quasi-isometric} if there exists a computable quasi-isometry between them.
    \end{definition}

    \begin{remark}
        For groups with decidable word problem, we have shown that any computable quasi-isometry $f\colon G \to H$ admits a computable quasi-isometry inverse. With a dovetailing argument, the proof of~\Cref{prop:QI-computableinverse} can be generalized to the case where $G$ is recursively presented and $H$ has decidable word problem. We do not know if the result is still true if both groups are merely recursively presented.
    \end{remark}

    By~\Cref{prop:QI-computableinverse}, it follows that being computably quasi-isometric is an equivalence relation for groups with decidable word problem. Putting~\Cref{lem:QI_medvedev} with the result of Cohen stating that for finitely presented groups and $X$ an SFT then $\operatorname{QI}[X]$ is an SFT, we obtain the following:
    
    \begin{corollary}\label{coro_QI}
        Let $G,H$ be finitely presented and computably quasi-isometric groups with decidable word problem. Then $\msft{G}=\msft{H}$.
    \end{corollary}

    \begin{proof}
        As the groups have decidable word problem there are computable quasi-isometries on both directions and thus it suffices to show that $\msft{H}\subset \msft{G}$.

        Let $X$ be an $H$-subshift. From~\Cref{lem:QI_medvedev} we obtain that $m(\operatorname{QI}[X]) = m(\operatorname{QI})\vee m(X)$. Furthermore, as there is a computable quasi-isometry, one can use it to construct a computable element of $\operatorname{QI}$ and thus it follows that $m(\operatorname{QI}) = 0_{\M}$. We obtain that for any subshift $X\subset A^H$ we have $m(\operatorname{QI}[X]) = m(X)$. Finally, as $G,H$ are finitely presented, it follows that for any SFT $X\subset A^H$ then $\operatorname{QI}[X]$ is a $G$-SFT with $m(\operatorname{QI}[X]) = m(X)$, therefore $\msft{H}\subset \msft{G}$. 
    \end{proof}

    For the next lemma we will require a few concepts from computable analysis that extend our definitions on $\{0,1\}^{\NN}$ to metric spaces. A \define{computable metric space} is a triple $(X,d,\mathcal{S})$ where $(X,d)$ is a separable metric space and $\mathcal{S}=(x_i)_{i \in \NN}$ is a countable dense sequence for which there exists a total computable map $f\colon \NN^3 \to \QQ$ such that $f(i,j,k)=q$ with the property that $|d(x_i,x_j)-q|\leq 2^{-k}$. A basic ball is a set of the form $B(x_i,q) = \{ x \in X : d(x,x_i)<q\}$ for some $i \in \NN$ and $q \in \QQ$. A set $U\subset X$ is called \define{effectively open} or $\Sigma_1^0$ if it can be written as a recursively enumerable union of basic balls. A map between computable metric spaces is \define{computable} if there is an algorithm which on input $(i,q) \in \NN \times \QQ$, returns a description of the inverse of $B(x_i,q)$ as a $\Sigma_1^0$ set.

    The Cantor set $A^{\NN}$ with $\mathcal{S}$ a lexicographic enumeration of eventually constant sequences and $d(x,y) = 2^{-\inf\{ n \in \NN : x_n \neq y_n \}}$ is an example of a computable metric space. Other examples of metric spaces that can be made computable are the Baire space $\NN^{\NN}$, the euclidean space $\mathbb{R}^n$ and the hyperbolic space $\mathbb{H}^n$.
    
    \begin{observation}
        Contrary to what happens between discrete groups, for computable metric spaces $X,Y$ there may exist a computable quasi-isometry from $X$ to $Y$, but no computable quasi-isometry from $Y$ to $X$. For instance, the identity from $\ZZ$ to $\RR$ is a computable quasi-isometry, but there exist no computable quasi-isometry from $\RR$ to $\ZZ$, because computable maps are continuous and thus constant.
    \end{observation}
    
    \begin{lemma}
           Let $(X,d,(x_i)_{i \in \NN})$ be a computable metric space and let $G,H$ be two finitely generated groups with decidable word problem. If there are computable quasi-isometries $f \colon G \to X$ and $h \colon H \to X$, then $G$ and $H$ are computably quasi-isometric.
    \end{lemma}

    \begin{proof}
        Fix finite set of generators $S,T$ for $G$ and $H$ respectively. Let $\widehat{f} \colon S^* \to X$ and $\widehat{h}\colon T^* \to X$ be computable maps that represent the quasi-isometries $f$ and $g$. Let $C\in \NN$ be a quasi-isometry constant for both $f$ and $g$. We construct $\widehat{\psi}\colon S^* \to T^*$ through the following algorithm.

        Let $w \in S^*$. Using the algorithm which decides the word problem of $G$, replace $w$ for a lexicographically minimal word $u$ that represents the same element of $G$. Through a dovetailing procedure, compute approximations of the recursively open sets $\widehat{f}^{-1}(B(x_i,C))$. Let $i_{w}$ be the first index for which this procedure returns that $u \in \widehat{f}^{-1}(B(x_{i_w},C))$. Next, let $v\in T^*$ be the first word which appears in the effectively open enumeration of $\widehat{h}^{-1}(B(x_{i_w},C))$ and set $\widehat{\psi}(w) = v$.

        From this description it is clear that $\widehat{\psi}$ is total computable and represents a map $\psi\colon G \to H$. Furthermore, it satisfies the property that for any $g \in G$ \[d(f(g),h(\psi(g))\leq d(f(g),x_{i_w}) + d(x_{i_w},h(\psi(g))) \leq 2C. \] 
        Where in the above inequality $w$ is any word in $S^*$ which represents $G$. It follows that if $\overline{h}$ is a coarse inverse for $h$, then \[ d_{\infty}( \overline{h}\circ f , \psi) \mbox{ is finite}.  \]
        And therefore $\psi$ is a quasi-isometry from $G$ to $H$.
    \end{proof}

    \begin{observation}\label{obs:svarcmilnor}A well-known source of quasi-isometries between groups and metric spaces is the \v{S}varc–Milnor Lemma. Let us note that this result can also be used to obtain computable quasi-isometries. We briefly recall the statement, and refer the reader to \cite[Theorem 23]{laharpe_Topics_2000} for details. Let $X$ be a metric space which is geodesic and proper, and let $G\curvearrowright X$ be a group action by isometries, such that the action is proper and the quotient $X/G$ is compact. Then $G$ is a finitely generated group, and quasi-isometric to $X$. Indeed, for all $x\in X$ the orbit map $f\colon G \to X$ given by $f(g) = g\cdot x$ is a quasi-isometry.

    Now suppose $G\curvearrowright X$ satisfies the hypotheses in the \v{S}varc–Milnor Lemma, that $(X,d,\mathcal{S})$ is a computable metric space, and that for all $g$ the map which sends $x\in X$ to $g\cdot x$ is computable. Then the orbit map of a computable point $x\in X$ provides a computable quasi-isometry from $G$ to $X$.
    \end{observation}

\section{Medvedev degrees of classes of subshifts on groups}\label{section:classification}

In these sections we will employ the machinery developed in this article to say as much as we can about the Medvedev degrees of subshifts in groups. 

\subsection{Known results}
        Here we review the state of the art of the classification of $\msft{G}$, $\msof{G}$ and $\meff{G}$ for different groups $G$. We begin by a rephrasing of Simpson's result in our language.
        \begin{theorem}[\cite{simpson_Medvedev_2014}]\label{simpson_result}
            Let $d\geq 2$. Then $\msft{\ZZ^d}=\msof{\ZZ^d}=\meff{\ZZ^d}$ equals the class of all $\Pi_1^0$ Medvedev degrees.
        \end{theorem}

        See \cite{durand_Fixedpoint_2012} for a different proof. On the other hand, it is well known that every SFT and sofic subshift on $\ZZ$ has finite orbits, and thus zero Medvedev degree. While it is not true that every SFT in the free group admits a finite orbit (see~\cite{piantadosi_Symbolic_2008}), it still contains computable configurations.
        \begin{proposition}\label{prop:virtually_free_have_0_degree}
            For every f.g virtually free group $G$ we have $\msft{G}=\msof{G}= \{0_{\M}\}$.
        \end{proposition}
        \begin{proof}
            By~\Cref{prop:commensurable_same_degrees} it suffices to consider a finitely generated free group. Let $F(S)$ be freely generated by a finite symmetric set $S$, and let $X\subset A^{F(S)}$ be a nonempty SFT. As $m(X)$ is a conjugacy invariant (\Cref{prop:basic-properties}), without loss of generality we can assume that $X$ is determined by patterns of the form $p\colon \{1,s\}\to A$, for $s\in S$ (see for instance~\cite[Proposition 1.6]{barbierilemp_thesis}). Let $A'\subset A$ be the set of all symbols in $A$ that occur in some configuration in $X$. Observe that for every $a\in A'$ and $s\in S$, there exists $b\in A'$ such that $1\mapsto a$, $s\mapsto b$ is not a forbidden pattern.
            
            We define a computable configuration $x\colon F(S)\to A'$ inductively. Fix a total order on $A'$, and let $x(1_{F(S)})$ be the minimal element in $A'$. Now, let $n >1$ and assume that we have defined $x(h)$ for all $h$ with $|h|_S < n$, and that no forbidden pattern occurs there. For $g$ with $|g|_S=n$, there is a unique $h$ with $|h|_S=n-1$ and $s \in S$ with $hs=g$. We define $x(g)$ as the minimal element in $A'$ such that the pattern $1_G\mapsto x(h),s\mapsto x(g)$ is not forbidden. Clearly no forbidden patterns can be created by this procedure because the Cayley graph of $F(S)$ with respect to $S$ is a tree, and thus we obtain that $x \in X$. This procedure is computable, and thus $m(X)=0_{\M}$. By~\Cref{prop:basic-properties} every topological factor of $X$ has at most degree $m(X) = 0_{\M}$ and the claim for sofic subshifts follows.            
        \end{proof}
        
        Let us now consider effective subshifts. Miller proved that $\meff{\ZZ}$ equals the class of all $\Pi_1^0$ Medvedev degrees \cite{miller_Two_2012}. This was then generalized by the second author to all groups with decidable word problem. 
        \begin{theorem}[\cite{carrasco-vargas_geometric_2023_}]\label{teo:nicarrasco_eff_pi01}
            Let $G$ be a group that is infinite and has decidable word problem. Then $\meff{G}$ equals the class of $\Pi_1^0$ Medvedev degrees.
        \end{theorem}

\subsection{Virtually polycyclic groups}

\begin{theorem}\label{thm:polycyclic}
    Let $G$ be a virtually polycyclic group. We have the following dichotomy
    \begin{enumerate}
        \item If $G$ is virtually cyclic, then $\msft{G} = \{0_\M\}$.
        \item If $G$ is not virtually cyclic, then $\msft{G}$ is the set of all $\Pi_1^0$ Medvedev degrees.
    \end{enumerate}
\end{theorem}

\begin{proof}
    Let $G$ be virtually cyclic. If $G$ is finite it is clear that $\msft{G}=\{0_{\M}\}$. If $G$ is infinite, then again by~\Cref{prop:commensurable_same_degrees} it follows that $\msft{G}=\msft{\ZZ} = \{0_\M\}$.

    Finally, suppose $G$ is a polycyclic group which is not virtually cyclic. Then it follows that $\ZZ^2 \leqslant G$ (see for instance~\cite{Segal_1983}). It is well known that polycyclic groups are recursively presented and have decidable membership problem (see for instance~\cite[Corollary 3.6]{Baumslag1981}), therefore by~\Cref{cor:subgroups_decidable_membership} we obtain that $\msft{\ZZ^2}\subset \msft{G}$. By Simpson's result (\Cref{simpson_result}) $\msft{\ZZ^2}$ is the set of all $\Pi_1^0$ Medvedev degrees, so $\msft{G}$ contains all $\Pi_1^0$ degrees. Conversely, every SFT on $G$ has a $\Pi_1^0$ Medvedev degree because the word problem of $G$ is decidable (\Cref{prop:msft-and-msoff-are-contained-in-Pi-1-degrees}). \end{proof}

\subsection{Direct products of groups}

We will show that in any direct product of two infinite and finitely generated groups with decidable word problem, every $\Pi_1^0$ Medvedev degree arises as the degree of an SFT. In order to do this, we shall use the following result of the second author.

\begin{theorem}[Theorem 1.7 of~\cite{carrasco-vargas_geometric_2023_}]\label{seward-calculable} Let $G$ be an infinite finitely generated group with decidable word problem. Then $G$ admits a computable translation-like action by $\ZZ$ with decidable orbit membership problem.
\end{theorem}

\begin{theorem}\label{thm:direct_products_clasification_medvedev_degrees}
        Let $H,K$ be two infinite and finitely generated groups with decidable word problem. Then $\msft{H\times K}$ is the set of all $\Pi_1^0$ Medvedev degrees.
\end{theorem}

\begin{proof}
    An immediate consequence of~\Cref{seward-calculable} is that $H\times K$ admits a computable translation-like action by $\ZZ^2$ which has decidable orbit membership problem 
    By \Cref{prop:translation-like-actions-and-medvedev-degrees}, we have that $\msft{\ZZ^2}$ is contained in $\msft{H\times K}$, and thus all $\Pi_1^0$ Medvedev degrees can be attained by an SFT on $H\times K$. On the other hand, the Medvedev degree of every SFT on $H\times K$ is a $\Pi_1^0$ degree because $H\times K$ has decidable word problem (\Cref{prop:msft-and-msoff-are-contained-in-Pi-1-degrees}).       
\end{proof}

    A class of groups where~\Cref{thm:direct_products_clasification_medvedev_degrees} can be applied meaningfully are branch groups. There is more than one definition in the literature of a branch group, see for instance~\cite{BartholdiGrigorchukSuni_branchgroups_2023}. However, both definitions have the following consequence: every branch group is commensurable to a direct product of two infinite groups. In particular, if $G$ is an infinite and finitely generated branch group with decidable word problem, then it is commensurable to the direct product of two infinite and finitely generated groups with decidable word problem. Putting together~\Cref{thm:direct_products_clasification_medvedev_degrees} and~\Cref{prop:commensurable_same_degrees} we obtain the following result.

\begin{corollary}\label{cor:branch_groups}
    For every infinite and finitely generated branch group $G$ with decidable word problem $\msft{G}$ is the set of all $\Pi_1^0$ Medvedev degrees. 
\end{corollary}

Without the assumption of decidable word problem, we can only show the existence of SFTs with nonzero Medvedev degree.

\begin{theorem}\label{thm:direct_products_nontrivial_medvedev_degrees}
        For any two infinite and finitely generated groups $H,K$ the set $\msft{H\times K}$ contains a nonzero Medvedev degree. Moreover, if $H,K$ are recursively presented, then $\msft{H\times K}$ contains the maximal $\Pi_1^0$ Medvedev degree.
\end{theorem}

\begin{proof}
    By Seward's result on translation-like actions~\cite[Theorem 1.4]{seward_Burnside_2014} we have that $\ZZ^2$ acts translation-like on $H \times K$. As $\ZZ^2$ is finitely presented and $\msft{\ZZ^2}$ is the set of all $\Pi_1^0$ Medvedev degrees,~\Cref{prop:translation-like-actions-and-medvedev-degrees} yields that there is a $(H \times K)$-SFT whose degree is at least the maximal $\Pi_1^0$ Medvedev degree, in particular it is  nonzero. If $H,K$ are recursively presented, then so is $H\times K$ and the upper bound also holds. 
\end{proof}

\begin{corollary}
    Every infinite and finitely generated branch group $G$ admits a nonempty SFT with nonzero Medvedev degree. If $G$ is recursively presented, then $\msft{G}$ contains the maximal $\Pi_1^0$ Medvedev degree.
\end{corollary}
  
\subsection{Groups which are quasi-isometric to the hyperbolic plane}\label{subsec:QI_hyperbolic}

In this section, we consider the hyperbolic plane $\mathbb{H} = \{(x,y) \in \mathbb{R}^2 : y >0 \}$ as a computable metric space given by its Riemannian metric and with dense countable set given by a computable enumeration of $\mathbb{Q}^2 \cap \mathbb{H}$.

\begin{theorem}\label{thm:hyperbolic_qi_medvedev}
    Let $G$ be a finitely generated group that is quasi-isometric to $\mathbb{H}$. Then there exists a $G$-SFT $X$ with $m(X)>0_{\M}$. 
\end{theorem}

\begin{proof}
    Any group which is quasi-isometric to a hyperbolic metric space is a word-hyperbolic group and thus is finitely presented and has decidable word problem (see for instance~\cite[Theorem 2.6]{bridson_Metric_1999}). In particular, all groups considered in the statement are finitely presented. As we already know that admitting an SFT with nonzero Medvedev degree is a quasi-isometry invariant for finitely presented groups (\Cref{admitir-sft-incalculable-es-invariante-de-qi}), it suffices to exhibit one group $G$ which is quasi-isometric to $\mathbb{H}$ and which admits a $G$-SFT $X$ with $m(X)>0_{\M}$.
    
    For this purpose, we consider the fundamental group $\pi_1(\Sigma_2)$ of a closed orientable surface of genus two. This group is finitely presented and quasi-isometric to the hyperbolic plane. In~\cite{domino-surface-groups}, the authors prove the undecidability of the domino problem for $\pi_1(\Sigma_2)$. The same construction, plus some results that are present in the literature, can be used to exhibit an SFT $X$ with $m(X)>0_{\M}$. 

Let us proceed more precisely. In~\cite[Corollary 2]{jeandel_Immortal_2012} the author shows that there exists a piecewise affine map $f\colon [0,1]^2 \to[0,1]^2$ with rational coefficients and endpoints, and such that every computable element in $[0,1]^2$ converges to $(0,0)$ in a finite number of iterations. However, the set of \define{immortal} points $I_f$ whose orbit does not converge is nonempty and thus contains only uncomputable points. Using a beautiful construction of Kari~\cite[Section 2]{kari_Tiling_2007}, the set $I_f$ can be encoded as a set of colorings of a graph which represents a binary tiling of $\mathbb{H}$ given by local rules. Finally, in~\cite[Section 5]{domino-surface-groups} the authors show that every such set of colorings can be encoded in a $\pi_1(\Sigma_2)$-SFT $X$. These constructions are effective in the sense that if $X$ had a computable point, then one could use it to obtain a computable element in $I_f$. Therefore $m(X)>0_{\M}$. \end{proof}

We remark that the corollary of~\cite{jeandel_Immortal_2012} that we use, is a dynamical analogue of a classical result of Hooper~\cite{hooper_undecidability_1966} which shows that the emptiness of the set of immortal configurations of a Turing machine is undecidable. We note that the technique of encoding Turing machines into piecewise linear maps was first used by Blondel et al ~\cite{BlondelBournezKoiran2001_piecewise} to show the undecidability of several dynamical properties of piecewise affine maps in $\RR^2$ and has since been used by several other authors. It is also a central piece in the proof of the undecidability of the domino problem in the hyperbolic plane see~\cite{kari_Tiling_2007}.

A natural class of examples is given by some Fuchsian groups. A group $G$ is \define{Fuchsian} if it is a discrete subgroup of $\operatorname{PSL(2,\RR)}$, a good introductory reference is~\cite{Katok1992-nk}. Fuchsian groups act properly discontinuously on $\mathbb{H}$ by isometries (M\"obius transformations). It follows by the Milnor-\v{S}varc lemma that if a Fuchsian group is finitely generated and acts co-compactly, it is quasi-isometric to $\mathbb{H}$.

\begin{corollary}
    Every finitely generated co-compact Fuchsian group has an SFT $X$ with $m(X)>0_{\M}$.
\end{corollary}
Natural examples of finitely generated co-compact Fuchsian groups are given by the hyperbolic triangle groups $\Delta(l,m,n)$ with $l,m,n \in \NN$ and $\frac{1}{l}+\frac{1}{n}+\frac{1}{m}<1$.

Now let us turn our attention to the classification of Medvedev degrees for groups which are quasi-isometric to $\mathbb{H}$. Recall that by~\Cref{coro_QI} all groups which admit a computable quasi-isometry to $\mathbb{H}$ form a computably quasi-isometry class. It is not hard to show that $\pi_1(\Sigma_2)$ belongs to this class, thus we obtain the following result.

\begin{corollary}
    Let $\pi_1(\Sigma_2)$ be the fundamental group of the closed orientable surface of genus two. If $G$ is a finitely generated group which admits a computable quasi-isometry $f\colon G \to \mathbb{H}$, then $\msft{G}=\msft{\pi_1(\Sigma_2)}$.
\end{corollary}

We remark that while we know that $\msft{\pi_1(\Sigma_2)}$ is non-trivial, we do not know what is the precise set of $\Pi_1^0$ degrees that can be attained. A difficulty here is that while it is known that there are Turing machines whose immortal set has nonzero Medvedev degree, it is not known whether all $\Pi_1^0$ Medvedev degrees can be obtained in this manner. On the other hand, it seems reasonable that a hierarchical construction such as the one of~\cite{goodman-strauss_hierarchical_2010} can be coupled with Simpson's construction to produce all $\Pi_1^0$ Medvedev degrees.

\begin{remark}
    We do not know if every finitely generated and co-compact Fuchsian group admits a computable quasi-isometry to $\mathbb{H}$. By~\Cref{obs:svarcmilnor}, a sufficient condition for a Fuchsian group to admit a computable quasi-isometry to $\mathbb{H}$ is that their generators can be represented by matrices in $\operatorname{PSL}_2(\RR)$ with computable coefficients. We do not know if this holds in general, but it certainly holds for well-known examples, such as the co-compact hyperbolic triangle groups.
\end{remark}

\subsection{Groups obtained through simulation results}\label{subsec:simulation}

Let $\rho \colon G \to H$ be an epimorphism. Recall that we denote the pullback of a subshift $X\subset A^{H}$, by $\rho^{*}(X)$. A group $G$ \define{simulates} $H$ (through $\rho$) if the pullback of any effective $H$-subshift is a sofic $G$-subshift\footnote{In other references, for instance~\cite{barbieri_carrasco_rojas_2024_effective,Barbieri_Sablik_Salo_2021}, the notion of simulation is more restrictive and asks that the pullback of every computable action of $H$ on any $\Pi_1^0$ set is the topological factor of some $G$-SFT.}. 
\begin{proposition}\label{prop:simulation}
    Let $G,H$ be finitely generated groups such that $G$ simulates $H$. The following statements hold: 
    \begin{enumerate}
        \item $\msof{G} \supset\meff{H}$.
        \item If $H$ is infinite and recursively presented, then $\msft{G}$ contains a Medvedev degree which is bounded below by the maximal $\Pi_1^0$ Medvedev degree.
        \item If $H$ is infinite and has decidable word problem, $\msof{G}$ contains the set of all $\Pi_1^0$ Medvedev degrees.
    \end{enumerate}
   \end{proposition} 
    \begin{proof}
        Let $X\subset A^{H}$ be any effective subshift.  As $G$ simulates $H$, we have that $\rho^*(X)$ is sofic. By~\Cref{lema:medvedev_degree_of_pullback} we have that $m(X) = m(\rho^*(X))$ and thus $\msof{G} \supset\meff{H}$.

        Now, let us suppose that $H$ is infinite and recursively presented. By~\Cref{cor:all_infinite_rp_groups_havenonzero_effective_medvedev_degree} we have that $\meff{H}$ contains the maximal $\Pi_1^0$ Medvedev degree, and thus by the first item, so does $\msof{G}$. If we let $Y$ be a sofic $G$-shift with this degree, any $G$-SFT extension $X$ will satisfy $m(X)\geq m(Y)$ by~\Cref{prop:basic-properties}.

        Finally, suppose that $H$ is infinite and with decidable word problem. By~\Cref{teo:nicarrasco_eff_pi01}, we have that $\meff{H}$ is the set of all $\Pi_1^0$ Medvedev degrees, and thus as $\msof{G} \supset\meff{H}$ we get that $\msof{G}$ also contains the set of all $\Pi_1^0$ Medvedev degrees.\end{proof}

\section{Further remarks and questions}\label{section:questions}

\subsection{Conjectures on the Medvedev degrees of classes of subshifts}

For every finitely generated and recursively presented group $G$, we have that $\meff{G}$ is contained in the set of $\Pi_1^0$ degrees (\Cref{prop:msft-and-msoff-are-contained-in-Pi-1-degrees}).  Furthermore, \Cref{cor:all_infinite_rp_groups_havenonzero_effective_medvedev_degree} shows that it always contains the maximal $\Pi_1^0$ degree. Thus the following question arises naturally:
    \begin{question}\label{question:effective_subshifts_degrees_on_rp}
        Let $G$ be a finitely generated group that is infinite and recursively presented. Is it true that effective subshifts on $G$ attain all $\Pi_1^0$ Medvedev degrees?
    \end{question}

In fact, with our current understanding, in every recursively presented group where $\msft{G}$ is known, it is either trivial or it coincides with the class of all $\Pi_1^0$ Medvedev degrees. This leads us to propose the following conjecture.

    \begin{conjecture}\label{conjecture_allSFTSarePi10completeclasses}
    Let $G$ be an infinite, finitely generated and recursively presented group. $G$ is not virtually free if and only if $\msft{G}$ is the set of all $\Pi_1^0$ Medvedev degrees.
\end{conjecture}

The condition that $G$ is recursively presented in~\Cref{conjecture_allSFTSarePi10completeclasses} is natural, as that gives the upper bound that $\msft{G}$ is contained in the set of all $\Pi_1^0$ Medvedev degrees, however, this is still wide open even in the case of groups with decidable word problem.
    
We remark that in the case of a recursively presented group, \Cref{prop:simulation} provides us with an SFT which achieves the maximal $\Pi_1^0$ Medvedev degree, but a priori it does not give us a tool to classify $\msft{G}$. Indeed, we can only conclude that for every sofic subshift $Y$ there exists a subshift of finite type with Medvedev degree \textit{at least} $m(Y)$. Although it seems unlikely, it might be the case that the lattice $\msft{G}$ may contain gaps which do not occur in $\msof{G}$.

\begin{question}\label{question:equal_degree_extension}
     Let $G$ be a recursively presented group and $Y$ a sofic $G$-subshift. Does $Y$ admit a $G$-SFT extension of equal Medvedev degree? In particular, is it true that $\msft{G}=\msof{G}$?
\end{question}

A positive answer to~\Cref{question:equal_degree_extension} would provide strong evidence towards~\Cref{conjecture_allSFTSarePi10completeclasses}.

\subsection{Immortal sets of Turing machines and hyperbolic groups}

In~\Cref{subsec:QI_hyperbolic} we mentioned a construction of Jeandel that produces a Turing machine which halts on every computable starting configuration, but loops in a $\Pi_1^0$ set of starting configurations with nonzero Medvedev degree. It is currently unknown if every $\Pi_1^0$ Medvedev degree can be obtained this way.

\begin{question}\label{question:immortality}
    Is it true that immortal sets of Turing machines attain all $\Pi_1^0$ Medvedev degrees?
\end{question}

If such a result were proven, it would imply by the same argument sketched in~\Cref{thm:hyperbolic_qi_medvedev} that the surface group $\pi_1(\Sigma_2)$ admits SFTs which attain all Medvedev degrees. In fact, it may be possible to adapt a construction of Bartholdi~\cite{bartholdi_2023_hyperbolic} to extend said result to all non virtually free hyperbolic groups. This would provide a way to settle~\Cref{conjecture_allSFTSarePi10completeclasses} for hyperbolic groups.





\subsection{Medvedev degrees, aperiodic subshifts, and the domino problem}
We now observe that for a fixed finitely generated group, the existence of SFTs with nonzero Medvedev degree implies both the existence of weakly aperiodic SFTs, and the undecidability of the domino problem.

\begin{proposition}\label{prop:med_implies_WA_conjecture}
        Let $G$ be a finitely generated group, and let $X$ be a nonempty $G$-subshift. If $m(X)>0_{\M}$, then $X$ is weakly aperiodic, that is, $G\curvearrowright X$ has no finite orbits. 
    \end{proposition}
    \begin{proof}
        We prove that if $X$ has a configuration with finite orbit, then $m(X)=0_{\M}$. Indeed, if $X$ has a configuration with finite orbit, then the same is true for the pullback $\widehat X$. Thus it suffices to prove that a configuration in $A^{F(S)}$ with finite orbit must be computable. 
        
        Indeed, if $x_0\in A^{F(S)}$ has finite orbit, then there is a finite group $K$, a group epimorphism $\phi\colon F(S)\to K$, and a configuration $y_0\in A^K$ such that the pullback $\phi^\ast (y_0)$ is equal to $x_0$ (where $\phi^\ast (y)(g)=y(\phi(g)),\ g\in F(S)$). As $A^K$ is finite it follows that $y_0$ is computable and thus that $x_0$ is computable: given $g\in F(S)$ we can compute $x_0(g)$ by first computing $\phi(g)\in K$ (recall that every homomorphism of finitely generated groups is computable), and setting $x_0(g)=y_0(\phi(g))$.\end{proof}

The domino problem of a group $G$ is the algorithmic problem of computing, from a finite set of pattern codings, if the associated $G$-SFT is nonempty. 
        
\begin{proposition}\label{prop:med_implies_domino}
    Let $G$ be a finitely generated group which admits a nonempty $G$-SFT $X$ with $m(X)>0_{\M}$. Then the domino problem for $G$ is undecidable. 
\end{proposition}
\begin{proof}
    It is well known that if $G$ has undecidable word problem, then $G$ has undecidable domino problem~\cite[Theorem 9.3.28]{aubrun_Domino_2018}. Therefore we will assume that $G$ has decidable word problem and perform computations directly on patterns instead of pattern codings. We will prove that if the domino problem for $G$ is decidable, and $X$ is a nonempty $G$-SFT on alphabet $A$, then $m(X)=0_{\M}$.

    Fix a word metric on $G$ and let $B_n=\{g\in G : |g|\leq n\}$. Given $L \subset A^{B_n}$, we denote by $Y(L)$ the $G$-SFT given by \[ Y(L) = \{ y \in A^G : (gy)|_{B_n}\in L \mbox{ for every } g \in G\}.\] 

    In other words, $Y(L)$ is the $G$-SFT defined by the forbidden patterns $\mathcal{F} = A^{B_n}\setminus L$. Fix $n_0 \in \NN$ and consider a sequence $(L_n)_{n\geq n_0}$ with the following properties: $L_n\subset A^{B_n}$ is nonempty, every pattern in $L_n$ appears in $X$, and finally, for every pattern $p\in L_n$, there exists a pattern $q\in L_{n+1}$ whose restriction to $B_n$ is $p$. We remark that if a computable sequence $(L_n)_{n \geq n_0}$ as above exists, then we can extract a computable configuration $x$ by letting $x|_{B_{n_0}} = p_{n_0} \in L_{n_0}$ arbitrary, and then inductively for $n>n_0$ we choose $p_n$ as the lexicographically minimal element such that $p_n|_{B_{n-1}} = x|_{B_{n-1}}$ and set $x|_{B_n} = p_n$. It is clear that $x$ is computable and $x \in \bigcap_{n \geq n_0} Y(L_n) \subset X$. It follows that if such a computable sequence exists, then $m(X)=0_{\M}$. 

    As $X$ is a $G$-SFT, there exists a smallest $n_0\in \NN$ for which there is $D\subset A^{B_{n_0}}$ with $X = Y(D)$. Let $L_{n_0}\subset D$ be minimal (for inclusion) such that $Y(L_{n_0})$ is nonempty. Notice that $Y(L_{n_0})\subset X$ and that every pattern in $L_{n_0}$ must appear in $Y(L_{n_0})$.
    
    Now we define the sequence. Set $L_{n_0}$ as above. For $n > n_0$ we perform the following recursive step. Suppose that $L_{n-1}$ has already been computed and that it is minimal such that $Y(L_{n-1})\subset X$ is nonempty. Compute the set $D_n$ of all $p \in A^{B_n}$ whose restriction to $B_{n-1}$ lies in $L_{n-1}$. Then clearly $Y(D_n) = Y(L_{n-1})\subset X$. Using the algorithm for the domino problem of $G$, compute $L_n\subset D_n$ which is minimal such that $Y(L_n)$ is nonempty. 

    This sequence $(L_n)_{n \geq n_0}$ is computable. We claim it satisfies the required properties. By construction, for each $n \geq n_0$, $L_n\subset A^{B_n}$ is nonempty. Minimality implies that all patterns in $L_{n}$ appear in $Y(L_{n})\subset X$. Finally, for every $p\in L_n$, there is $q\in L_{n+1}$ whose restriction to $B_{n}$ equals $p$. Indeed, if this was not the case, then the set $R_{n}$ defined by restricting all patterns in $L_{n+1}$ to $B_n$ would be properly contained in $L_n$, and $\emptyset \neq Y(L_{n+1})\subset Y(R_n)$, contradicting the minimality of $L_{n}$. \end{proof}

These results lead us to conjecture the following:
\begin{conjecture}\label{conj:uncomputableconfig}
    Let $G$ be a finitely generated group. $G$ is virtually free if and only if every nonempty $G$-SFT has Medvedev degree $0_{\M}$.
\end{conjecture}

\begin{observation}\label{Medvedev-conjecture-is-pulent}
    Note that \Cref{conj:uncomputableconfig} implies Carroll and Penland's conjecture \cite{carroll_Periodic_2015} that all infinite finitely generated groups that are not virtually cyclic admit weakly aperiodic SFTs. This simply follows from \Cref{prop:med_implies_WA_conjecture}, plus the fact that all virtually free groups that are not virtually cyclic admit weakly aperiodic SFTs \cite{piantadosi_Symbolic_2008}.

    We also note that \Cref{conj:uncomputableconfig} implies Ballier and Stein's conjecture that a finitely generated infinite group has decidable domino problem if and only if it is virtually free~\cite{ballier_domino_2018}. Indeed, this follows from~\Cref{prop:med_implies_domino} and the fact that the conjecture holds for virtually free groups.
\end{observation}

\subsection{Beyond $\Pi_1^0$ Medvedev degrees}
Most of our results apply to recursively presented groups, where $\msft{G}$ is contained in the set of $\Pi_1^0$ degrees. It is natural to ask what happens beyond recursively presented groups. In this case we can show that for a group whose word problem is too complex, every strongly aperiodic $G$-subshift (one such that the shift action is free) has a Medvedev degree beyond $\Pi_1^0$ degrees:
\begin{proposition}\label{prop:beyondtimeandspace}
Let $G$ be a group whose word problem is not co-recursively enumerable. For every nonempty strongly aperiodic $G$-subshift $X$ we have that $m(X)$ is not a $\Pi_1^0$ Medvedev degree.
\end{proposition}
\begin{proof}
    Suppose that $G$ admits a strongly aperiodic subshift $X$ such that $m(X)$ is a $\Pi_1^0$ degree. We will prove that then $G$ has a co-recursively enumerable word problem. 

    Let $S$ be a symmetric set of generators for $G$. Now we will apply well-known results from computable analysis to the computable metric space $A^{F(S)}$, the reader is referred to ~\cite[Section 3]{barbieri_carrasco_rojas_2024_effective}. We start by proving that the pullback subshift $\widehat X\subset A^{F(S)}$ contains a nonempty subset that is $\Pi_1^0$. Indeed, as $m(X)$ is a $\Pi_1^0$ degree, there is a $\Pi_1^0$ set $P\subset \{0,1\}^\NN$ and a computable function $f\colon P\to \widehat X$. As the spaces $A^\NN$ and $A^{F(S)}$ are recursively compact, the computable image of a $\Pi_1^0$ set must be a $\Pi_1^0$ set. It follows that $Y=f(P)$ is a $\Pi_1^0$ subset of $\widehat{X}$. 
    
    For $w \in F(S)$, consider the stabilizer $\operatorname{Fix}(w) = \{ x \in A^{F(S)} : wx = x\}$, and observe that it is a $\Pi_1^0$ set. As $X$ is strongly aperiodic, a word $w\in S^\ast$ satisfies $\underline{w}\ne 1_G$ if and only if $\operatorname{Fix}(w)\cap Y$ is empty. As both $Y$ and $\operatorname{Fix}(w)$ are $\Pi_1^0$ sets, it follows that $\operatorname{Fix}(w)\cap Y$ is $\Pi_1^0$. Finally, as $A^{F(S)}$ is recursively compact, the collection of descriptions of $\Pi_1^0$ sets which are empty is recursively enumerable (see~\cite[Remark 3.13]{barbieri_carrasco_rojas_2024_effective}) and thus this gives an algorithm to enumerate the $w \in S^{*}$ which do not represent $1_{G}$.
\end{proof}

We remark that in the case where $G$ is a finitely generated and recursively presented group which admits a strongly aperiodic effective $G$-subshift, then~\Cref{prop:beyondtimeandspace} implies that $G$ has decidable word problem. This recovers a result by Jeandel~\cite[Corollary 2.7]{jeandel_notes_subshift_groups}.

We also remark that using simulation theorems it is possible to construct groups with non co-recursively enumerable word problem and which admit strongly aperiodic SFTs. This in particular shows that there exist finitely generated groups $G$ for which $\msft{G}$ is not contained in the $\Pi_1^0$ degrees. This result will appear in an upcoming article of the first author and Salo.

\Addresses

\bibliographystyle{abbrv}
\bibliography{bibliografia,MyLibrary}
\end{document}